\documentclass[11pt,a4paper]{article}
\usepackage[nottoc]{tocbibind}
\usepackage{graphicx}
\usepackage{enumerate}
\usepackage{amsmath}
\usepackage{amsfonts}
\usepackage{psfrag}
\usepackage{color}
\usepackage{pgf,tikz}
\usepackage{mathrsfs}
\usetikzlibrary{arrows}
\usepackage{fancyhdr}
\usepackage{txfonts} 
\usepackage{dutchcal}

\usepackage[colorlinks,pdfpagelabels,pdfstartview = FitH,bookmarksopen = true,bookmarksnumbered = true,linkcolor = blue,plainpages = false,hypertexnames = false,citecolor = red] {hyperref}

\newtheorem{theorem}{Theorem}

\newtheorem{corollary}[theorem]{Corollary}
\newtheorem{definition}{Definition}
\newtheorem{remark}{Remark}

\newtheorem{lemma}{Lemma}
\newtheorem{proposition}{Proposition}

\newcommand{\mint}{\mathop{\int\hskip -1,05em -\, \!\!\!}\nolimits}

\numberwithin{equation}{section}
\allowdisplaybreaks[4]

\newenvironment{proof}{\smallskip\noindent\emph{Proof.}\hspace{1pt}}%
{\hspace{-5pt}{\nobreak\quad\nobreak\hfill\nobreak$\square$\vspace{8pt}%
\par}\smallskip\goodbreak}

\newcommand{\g}{\mathcal{G}}
\newcommand{\Ruz}{R{\accent'27u}{\accent'24z}i{\accent'24c}ka}

\newcommand{\m}{\mathcal{M}}
\newcommand{\reach}{\mathrm{reach}}

\newcommand\eps\varepsilon
\def\eqn#1$$#2$${\begin{equation}\label#1#2\end{equation}}

\newcommand{\be}{\begin{equation}}
\newcommand{\ee}{\end{equation}}

\newcommand{\rr}{\varrho}

\newcommand{\dist}{\operatorname{dist}}
\newcommand{\const}{\operatorname{const}}
\newcommand{\snr}[1]{\lvert #1\rvert}

\newcommand{\nr}[1]{\lVert #1 \rVert}

\newcommand{\stackleqq}[1]{\stackrel{\rif{#1}}{ \leq}}

\newcommand{\RN}{\mathbb{R}^{N}}
\newcommand{\SN}{\mathbb{S}^{N-1}}
\newcommand{\N}{\mathbb{N}}

\def\name[#1, #2]{#1 #2}

\newcommand{\rif}[1]{(\ref{#1})}

\newcommand{\stackleq}[1]{\stackrel{\rif{#1}}{ \leq}}

\newcommand{\U}{U}

\definecolor{ffqqqq}{rgb}{1.,0.,0.}
\definecolor{uuuuuu}{rgb}{0.26666666666666666,0.26666666666666666,0.26666666666666666}

\DeclareMathOperator{\diam}{diam}

\begin{document}
\title{Partial regularity for manifold constrained $p(x)$-harmonic maps}
\author{
\textsc{Cristiana De Filippis\thanks{Mathematical Institute, University of Oxford, Andrew Wiles Building, Radcliffe Observatory Quarter, Woodstock Road, Oxford, OX26GG, Oxford, United Kingdom. E-mail:
      \texttt{Cristiana.DeFilippis@maths.ox.ac.uk}} }}
     
\date{\ 
}
\maketitle

\thispagestyle{plain}
\vspace{-1.5cm}
\begin{abstract}
We prove that manifold constrained $p(x)$-harmonic maps are locally $C^{1,\beta_{0}}$-regular outside a set of zero $n$-dimensional Lebesgue's measure, for some $\beta_{0} \in (0,1)$. We also provide an estimate from above of the Hausdorff dimension of the singular set.
\end{abstract}
\setlength{\voffset}{-0in} \setlength{\textheight}{0.9\textheight}

\setcounter{page}{1} \setcounter{equation}{0}
\tableofcontents

\section*{Introduction}
We prove local $C^{1,\beta_{0}}$-partial regularity for manifold constrained $p(x)$-harmonic maps. More precisely, we consider local minimizers of the functional
\begin{flalign}\label{cvp}
W^{1,p(\cdot)}(\Omega,\m) \ni w \mapsto \mathcal{E}(w,\Omega):= \int_\Omega k(x)\snr{Dw}^{p(x)} \ dx\;, 
\end{flalign}
where $p(\cdot)$ and $k(\cdot)$ are H\"older continuous functions (see $(\mathrm{P1})$-$(\mathrm{P2})$ and $(\mathrm{K1})$-$(\mathrm{K2})$ below for the precise assumptions), $\Omega \subset \mathbb R^n$, $n\ge 2$, is an open, bounded subset and $\m\subset \RN$, $N\ge 3$, is an $m$-dimensional, compact submanifold endowed with a suitable topology. We refer to Section \ref{S1} for the precise notation. Our final outcome is that there exists a relatively open set $\Omega_{0}\subset \Omega$ of full $n$-dimensional Lebesgue measure such that $u \in C^{1,\beta_{0}}_{\mathrm{loc}}(\Omega_{0},\m)$ for some $\beta_{0}\in (0,1)$ and $\Sigma_{0}(u):=\Omega\setminus \Omega_{0}$ has Hausdorff dimension at the most equal to $n-\gamma_{1}$. Moreover, after imposing some extra restrictions on the variable exponent $p(\cdot)$, we are able to provide a further reduction to the Hausdorff dimension of the singular set of $\m$-constrained minimizers of the $p(\cdot)$-energy
\begin{flalign}\label{modello}
w\mapsto \int_\Omega \snr{Dw}^{p(x)} \ dx.
\end{flalign}
Let us put our results into the context of the available literature. Functionals with variable growth exponent modelled on the one in \rif{modello} have been introduced in the setting of Calculus of Variations and Homogenization in the fundamental works of Zhikov \cite{zhi1, zhi2, zhi3, zhi4}. Energies as in \rif{modello} also occur in the modelling of electro-rheological fluids, a class of non-newtonian fluids whose viscosity properties are influenced by the presence of external electromagnetic fields \cite{acemin, RR}. As for regularity, the first result in the vectorial case has been obtained by Coscia \& Mingione in \cite{cosmin}, where it is shown that local minimizers of energy \rif{modello} are locally $C^{1,\beta}$-regular in the unconstrained case. This is the optimal generalization of the classical results of Uhlenbeck concerning the standard case when $p(\cdot)$ is a constant. We refer to \cite{KM1, KM2, Manth, dark, uhl, ura} for a survey of regularity results in the $p$-growth case, both for scalar and vector valued minimizers. Subsequently, the regularity theory of functionals with variable growth has been developed in a series of interesting papers by Ragusa, Tachikawa and Usuba \cite{RT1, RT2, RT3, tac, TU}, where the authors established partial regularity results for unconstrained minimizers that are on the other hand obviously related to the constrained case. Especially, in \cite{tac} Tachikawa gives an interesting partial regularity result and some singular set estimates for a class of functionals related to the constrained minimization problem in which minimizers are assumed to take values in a single chart. This generalizes the well-known results of Giaquinta \& Giusti \cite{giagiu} valid in the case of quadratic functionals with special structure.  In this paper we finally tackle the case of local minimizers with values into a manifold, provided that suitable topological assumptions are considered on the manifold $\m$ and optimal regularity conditions are in force on $p(\cdot)$ and $k(\cdot)$. Our first main result is the following:
\begin{theorem}\label{T0}
Let $u \in W^{1,p(\cdot)}(\Omega, \m)$ be a local minimizer of the functional in \eqref{cvp}, where $p(\cdot)$ satisfies assumptions $(\mathrm{P1})$-$(\mathrm{P2})$, $k(\cdot)$ satisfies $(\mathrm{K1})$-$(\mathrm{K2})$ and $\m$ is as in $(\mathrm{\m 1})$-$(\mathrm{\m 2})$. Then there exists a relatively open set $\Omega_{0}\subset \Omega$ such that $u \in C^{1,\beta_{0}}_{\mathrm{loc}}(\Omega_{0},\m)$ for some $\beta_{0}\in (0,1)$, and $\mathcal{H}^{n-\gamma_{1}}(\Omega \setminus \Omega_{0})=0$.
\end{theorem}
By strengthening further the assumptions on the variable exponent $p(\cdot)$, we are then able to provide a better dimension estimate for the singular set. This is in the following:
\begin{theorem}\label{T1}
Let $u \in W^{1,p(\cdot)}(\Omega,\m)$ be a constrained local minimizer of energy \eqref{modello}, 
where $p(\cdot)\in Lip(\Omega)$, $\gamma_{1}\ge 2$ and $\m$ is as in $(\mathrm{\m 1})$-$(\mathrm{\m 2})$. Then,
\begin{itemize}
    \item[i.] if $n\le [\gamma_{1}]+1$, then $u$ can have only isolated singularities;
    \item[ii.] if $n>[\gamma_{1}]+1$, then the Hausdorff dimension of the singular set is at the most $n-[\gamma_{1}]-1$.
\end{itemize}
\end{theorem}
As they are stated, our results are the natural generalization of the classical ones in \cite{harkinlin, harlin, luc, shouhl} for the case $p(\cdot)\equiv$ constant. For the vectorial quasiconvex case with standard $p$-growth we refer to the recent work of Hopper \cite{hop}. The extension we make here to the variable exponent case requires a number of non-trivial additional ideas and tools, especially, as far as the dimension estimates stated in Theorem \ref{T1} are concerned. This is also related to the recent, aforementioned paper of Tachikawa \cite{tac}, and it is based on the use of a suitable monotonicity formula. We remark that the variable exponent functional in \rif{cvp} is a significant instance of functionals with $(p,q)$-growth (following the terminology introduced by Marcellini, \cite{mar1, mar2}). These are variational integrals of the type $w\mapsto \int F(x,Dw) \ dx$, where the integrand $F(\cdot)$ satisfies
\begin{flalign*}
\snr{z}^{p}\lesssim F(x,y,z)\lesssim (1+\snr{z}^{q})\,, \qquad 1< p\leq q\;.
\end{flalign*}
The study of such functionals has undergone an intensive development over the last years, see for instance \cite{carkrispas, diestrver1,sharp, lie, mar1, mar2, dark}. Another prominent model in this class is the so called Double Phase energy, where
\begin{flalign*}
F(x,z)=\snr{z}^{p}+a(x)\snr{z}^{q}\;, \qquad 0\leq a(x)\leq L\;.
\end{flalign*}
This model shares several features with the variable growth exponent and has been again  introduced by Zhikov in \cite{zhi3}. Indeed, here once again the growth exponent with respect to the gradient variable is determined by the space variable $x$, since the ellipticity type changes according to the the positivity of the coefficient $a(\cdot)$. There are several analogies between the variable exponent energy and the double phase one. In particular, one should notice the similarities between the use of the Gehring's Lemma-based reverse H\"older inequalities made here and the reverse H\"older inequality coming from fractional differentiability exploited in \cite{colmin, colmin2}. Moreover, compare the use of localization methods based on $p$-harmonic type approximation implemented here and in \cite{barcolmin}. Such analogies point to a unified approach to non-autonomous functionals with $(p,q)$-growth conditions, partially implemented in \cite{me}. We plan to investigate this in the context of constrained minimizers in a forthcoming paper \cite{demi}. 

\section{Notation, main assumptions and functional setting}\label{S1}
Throughout this paper, $\Omega$ denotes an open, bounded subset of $\mathbb{R}^{n}$ with $n\ge 2$ and our target space will be a submanifold of $\RN$, $N\ge 3$. As usual, we denote by $c$ a general constant larger than one. Different occurrences from line to line will be still indicated by $c$ and relevant dependencies from certain parameters will be emphasized using brackets, i.e.: $c(n,p)$ means that $c$ depends on $n$ and $p$. We denote $B_{r}(x_{0}):=\left\{x \in \mathbb{R}^{n}\colon \snr{x-x_{0}}<r\right\}$ the open ball centered in $x_{0}$ with radius $r>0$; when not relevant, or clear from the context, we shall omit indicating the center: $B_{r}\equiv B_{r}(x_{0})$. Moreover, for integer $k\ge 1$, by $\omega_{k}$ we mean the $k$-dimensional Lebesgue measure of the unit ball $B_{1}(0)\subset \mathbb{R}^{k}$. Along the paper, $k$ will assume values $N$ or $m$. When referring to balls in $\RN$, we will stress it with the apex $N$, i.e.: $B_{r}^{N}(a_{0})$ is the open ball with center $a_{0}\in \RN$ and positive radius $r$. For $\alpha,\beta \in \{1,\cdots,n\}$ and $i,j\in\{1,\cdots,N\}$, we set $\delta^{\alpha\beta}\equiv0$, $\delta_{ij}\equiv 0$ if $\alpha \not =\beta$, $i\not =j$ respectively and $\delta^{\alpha\alpha}=\delta_{ii}\equiv1$. With $U\subset \mathbb{R}^{n}$ being a measurable subset with positive, finite Lebesgue measure $0<\snr{U}<\infty$ and with $f\colon U\to \mathbb{R}^{k}$ being a measurable map, we shall denote by
\begin{flalign*}
(f)_{U}:= \mint_{U}f(x) \ dx =\frac{1}{\snr{U}}\int_{U}f(x) \ dx
\end{flalign*}
its integral average. In particular, when $U\equiv B_{r}(x_{0})$, we will indicate only the radius and, if necessary, the centre of the ball, i.e.: $(f)_{r}\equiv (f)_{r,x_{0}}:=(f)_{B_{r}(x_{0})}$. For $g\colon \Omega\to \mathbb{R}^{k}$ and $U\subset \Omega$, with $\gamma \in (0,1]$ being a given number we shall denote
\begin{flalign*}
[g]_{0,\gamma;U}:=\sup_{x,y \in U; x\not=y}\frac{\snr{g(x)-g(y)}}{\snr{x-y}^{\gamma}}, \quad [g]_{0,\gamma}:=[g]_{0,\gamma;\Omega}.
\end{flalign*}
It is well known that the quantity defined above is a seminorm and when $[g]_{0,\gamma;U}<\infty$, we will say that $g$ belongs to the H\"older space $C^{0,\gamma}(U,\mathbb{R}^{k})$. Let us turn to the main assumptions that will characterize our problem. When considering the functional in \eqref{cvp}, the exponent $p(\cdot)$ will always satisfy
\begin{flalign*}
&(\mathrm{P1}) \ p\in C^{0,\alpha}(\Omega), \ \ \alpha \in (0,1],\\
&(\mathrm{P2}) \ 1<\gamma_{1}\le p(x)\le \gamma_{2}<\infty \ \ \mathrm{for \ all \ }x \in \Omega,\\
& \gamma_{1}:=\inf_{x\in \Omega}p(x) \quad \mathrm{and}\quad \gamma_{2}:=\sup_{x\in \Omega}p(x),
\end{flalign*}
while the coefficient $k(\cdot)$ is so that
\begin{flalign*}
&(\mathrm{K1}) \ k\in C^{0,\nu}(\Omega), \ \ \nu \in (0,1],\\
&(\mathrm{K2}) \ 0<\lambda\le k(x)\le \Lambda<\infty \ \ \mathrm{for \ all \ }x \in \Omega,
\end{flalign*}
hold true. Clearly, in hypotheses $\mbox{(P1)}$-$\mbox{(K1)}$ there is no loss of generality in supposing $\alpha=\nu$, since in the forthcoming estimates only $\min\left\{\alpha,\nu\right\}$ will be relevant, so, for simplicity, from now on we will assume $p(\cdot)$, $k(\cdot)\in C^{0,\alpha}(\Omega)$. These assumptions are optimal in order to get local H\"older continuity for the gradient of a minimizer of problem \eqref{cvp}. This is evident already in the scalar, linear case, (Schauder estimates). For any given ball $B_{r}\Subset \Omega$, we denote
\begin{flalign}\label{p1p2}
p_{1}(r):=\inf_{x \in B_{r}}p(x) \quad \mbox{and}\quad p_{2}(r):=\sup_{x \in B_{r}}p(x).
\end{flalign}
Notice that there is no loss of generality in assuming $\gamma_{1}<\gamma_{2}$, otherwise $p(\cdot)\equiv \const$ on $\Omega$, and in this case the problem is very well understood, \cite{giamar,harkinlin,harlin,luc,shouhl,sim}. Furthermore, we need to impose some topological restriction on the manifold $\m$. Precisely, we ask that
\begin{flalign*}
&(\mathrm{\mathcal{M}1}) \ \m \ \mathrm{is \ a \ compact, \ }m-\mathrm{dimensional}, \ C^{3} \ \mathrm{Riemannian \ submanifold \ without \ boundary \ of \ }\RN,\\
&(\mathrm{\mathcal{M}2}) \ \m \ \mathrm{is} \  [\gamma_{2}]-1 \ \mathrm{connected}.
\end{flalign*}
Here $[x]$ denotes the integer part of $x$. We refer to Section \ref{pre} for a detailed description of the geometry of $\m$. Finally, for shorten the notation we shall collect the main parameters of the problem in the quantity 
\begin{flalign*}
\texttt{data}:=(n,N,\m,\lambda,\Lambda,\gamma_{1},\gamma_{2},[k]_{0,\alpha},[p]_{0,\alpha},\alpha).
\end{flalign*}
As to fully clarify the framework we are going to adopt, we need to introduce some basic terminology on the so-called Musielak-Orlicz-Sobolev spaces. Essentially, these are Sobolev spaces defined by the fact that the distributional derivatives lie in a suitable Musielak-Orlicz space, rather than in a Lebesgue space as usual. Classical Sobolev spaces are then a particular case. Such spaces and related variational problems are discussed for instance in \cite{harhaskle, zhi4}, to which we refer for more details. Here, we will consider spaces related to the variable exponent case in both unconstrained and manifold-constrained settings. 
\begin{definition} \label{N11}
Given an open set $\Omega \subset \mathbb{R}^{n},$ the Musielak-Orlicz space $L^{p(\cdot)}(\Omega,\mathbb{R}^{k})$, $k\ge 1$, with $p(\cdot)$ satisfying $\mathrm{(P1)}$-$\mathrm{(P2)}$, is defined as
\begin{flalign*}
L^{p(\cdot)}(\Omega,\mathbb{R}^{k}):=\left \{ \ w\colon \Omega \to \mathbb{R}^{k} \ \mathrm{measurable \ and \ }\int_{\Omega}\snr{w}^{p(x)}\ dx <\infty \ \right\},
\end{flalign*}
and, consequently,
\begin{flalign*}
W^{1,p(\cdot)}(\Omega,\mathbb{R}^{k}):=\left \{ \ w\in W^{1,1}(\Omega,\mathbb{R}^{k})\cap L^{p(\cdot)}(\Omega,\mathbb{R}^{k})\ \mathrm{such \ that \ }Dw \in L^{p(\cdot)}(\Omega, \mathbb{R}^{k\times n}) \ \right\}.
\end{flalign*}
The variants $W^{1,p(\cdot)}_{0}(\Omega,\mathbb{R}^{k})$ and $W^{1,p(\cdot)}_{\mathrm{loc}}(\Omega,\mathbb{R}^{k})$ are defined in an obvious way.
\end{definition}
It is well known that, under assumptions $(\mathrm{P1})$-$(\mathrm{P2})$, the set of smooth maps is dense in $W^{1,p(\cdot)}(\Omega,\mathbb{R}^{k})$, see e.g. \cite{sharp,zhi1,zhi3,zhi4}. Following \cite{me,hop} we also recall the analogous definition of such spaces when mappings take values into $\m$.
\begin{definition}\label{N12}
Let $\m$ be a compact submanifold of $\mathbb{R}^{k}$, $k\ge 3$, without boundary and $\Omega \subset \mathbb{R}^{n}$ an open set. For $p(\cdot)$ satisfying $\mbox{(P1)}$-$(\mbox{P2})$, the Musielak-Orlicz-Sobolev space $W^{1,p(\cdot)}(\Omega,\m)$ of functions into $\m$ can be defined as
\begin{flalign*}
W^{1,p(\cdot)}(\Omega,\m):=\left\{ \ w \in W^{1,p(\cdot)}( \Omega,\mathbb{R}^{k})\colon \ w(x) \in \m \ \mathrm{for \ a.\ e.} \ x \in \Omega \ \right\}.
\end{flalign*}
The local space $W^{1,p(\cdot)}_{\mathrm{loc}}(\Omega,\m)$ consists of maps belonging to $W^{1,p(\cdot)}(U,\m)$ for all open sets $U\Subset \Omega$.
\end{definition}
When $\mbox{(P1)}$-$\mbox{(P2)}$ and $(\mathrm{\mathcal{M}1})$-$(\mathrm{\mathcal{M}2})$ are in force, a quick modification of \cite[Lemma 6]{me} shows that Lipschitz maps are dense in $W^{1,p(\cdot)}(\Omega,\m)$ as well. Of course, when $p(\cdot)\equiv \const$, Definitions \ref{N11} and \ref{N12} reduce to the classical Sobolev spaces $W^{1,p}(\Omega,\mathbb{R}^{k})$ and $W^{1,p}(\Omega,\m)$ respectively. Owing to the $p(\cdot)$-growth behavior of the integrand in \eqref{cvp}, we display our definition of local minimizer.
\begin{definition}\label{D1}
A map $u\in W^{1,p(\cdot)}_{\mathrm{loc}}(\Omega,\m)$ is a constrained local minimizer of the functional $\mathcal{E}(\cdot)$ defined in \eqref{cvp} if and only if
\begin{flalign*}
x\mapsto k(x)\snr{Du(x)}^{p(x)}\in L^{1}_{\mathrm{loc}}(\Omega) \quad \mathrm{and}\quad \int_{U}k(x)\snr{Du}^{p(x)} \ dx\le \int_{U}k(x)\snr{Dw}^{p(x)} \ dx,
\end{flalign*}
for all open sets $U\Subset \Omega$ and all $w\in W^{1,p(\cdot)}_{u}(U,\m)$, where
\begin{flalign*}
W^{1,p(\cdot)}_{u}(U,\m):=\left(u+W^{1,p(\cdot)}_{0}(\U, \mathbb{R}^{N})\right)\cap W^{1,p(\cdot)}(\U,\m).
\end{flalign*}
\end{definition}
In Definition \ref{D1}, local minimizers are given as maps belonging to the local space $W^{1,p(\cdot)}_{\mathrm{loc}}(\Omega,\m)$. We stress that, since all the regularity properties of constrained local minimizers treated in this work are of local nature, there is no loss of generality in assuming that $u \in W^{1,p(\cdot)}(\Omega,\m)$ and that $x\mapsto k(x)\snr{Du(x)}^{p(x)}\in L^{1}(\Omega)$, see the statements of Theorems \ref{T0}-\ref{T1}.
\begin{remark}
\emph{By continuity, all the constants depending on certain fixed values of the map $p(\cdot)$ are stable when $p(\cdot)$ varies in the interval $[\gamma_{1},\gamma_{2}]$. Thus, whenever a constant depends on some $p \in [\gamma_{1},\gamma_{2}]$, this dependence will be denoted by only mentioning $\gamma_{1}$ and $\gamma_{2}$, i.e.: $c(p)\equiv c(\gamma_{1}, \gamma_{2})$.}
\end{remark}

\section{Preliminaries}\label{pre}
We shall split this section into two parts. In the first one, we collect some basic results concerning the regularity of minimizers of certain type of functionals and in the second one we will give a detailed description of the topology of $\m$, together with some extension lemmas, which will turn fundamental in order to construct suitable comparison maps in some steps of the proofs of Theorems \ref{T0} and \ref{T1}.
\subsection{Known regularity results}
We start by reporting a Lipschitz estimate for the gradient and a decay estimate for the excess functional of unconstrained local minimizers of functionals of the $p$-laplacean type.
\begin{proposition}\cite{acefus1,giagiudiff,giamod}\label{rif}
Let $\Omega\subset \mathbb{R}^{n}$ be an open, bounded set, $p\in (1,\infty)$, $0<\lambda<\Lambda $ and $0<\mathcal{l}<L$ be absolute constants, $g^{\alpha\beta}$ and $h_{ij}$ be constant matrices, uniformly elliptic in the sense that
\begin{flalign*}
\mathcal{l}\snr{\xi}^{2}\le g^{\alpha\beta}\xi_{\alpha}\xi_{\beta}\le L\snr{\xi}^{2} \ \ \mbox{for all }\xi \in \mathbb{R}^{n}\quad \mbox{and}\quad \mathcal{l}\snr{\eta}^{2}\le h_{ij}\eta^{i}\eta^{j}\le L\snr{\eta}^{2} \ \ \mbox{for all }\eta \in \mathbb{R}^{k},
\end{flalign*}
uniformly bounded, $[g^{\alpha\beta}]=[g^{\alpha\beta}]^{-1}$; $\alpha,\beta \in \{1,\cdots,n\}$, $i,j\in \{1,\cdots,k\}$. Then, if $v\in W^{1,p}(\Omega,\mathbb{R}^{k})$ is a local minimizer of the integral functional
\begin{flalign}\label{rev1}
W^{1,p}(\Omega,\mathbb{R}^{k})\ni w\mapsto\mathcal{F}(w,\Omega):=\int_{\Omega}k_{0}\left(g^{\alpha\beta}h_{ij}D_{\alpha}w^{i}D_{\beta}w^{j}\right)^{p/2} \ dx,
\end{flalign}
where $k_{0}\in [\lambda,\Lambda]$ is a constant, then for all $B_{\varrho}\subset B_{r}\Subset \Omega$ the following reference estimates hold:
\begin{flalign}\label{140}
\mint_{B_{\varrho}}\snr{Dv}^{p} \ dx \le c \mint_{B_{r}}\snr{Dv}^{p} \ dx \quad \mathrm{and} \quad \mint_{B_{\varrho}}\snr{Dv-(Dv)_{\varrho}}^{p} \ dx \le c(\varrho/r)^{\mu p}\mint_{B_{r}}\snr{Dv}^{p} \ dx,
\end{flalign}
for $c=c(n,k,\lambda,\Lambda,\mathcal{l},L,p)$ and $\mu=\mu(n,k,\lambda,\Lambda,\mathcal{l},L,p)$. 
\end{proposition}
The next is a $p$-harmonic approximation lemma, which will play a crucial role in the proof of Theorem \ref{T0}. We will state it in the form which better fits our necessities.
\begin{lemma}{\cite{diestrver}}\label{phihar}
Let $\Omega\subset \mathbb{R}^{n}$ be an open subset and $p\in (1,\infty)$. For every $\tilde{\theta}>0$ and $\tilde{d} \in (0,1)$ there exists $\tilde{\delta}>0$ depending only on $\tilde{\theta}$, $\tilde{d}$, $p$, such that the following holds. Let $B_{r}\subset \mathbb{R}^{n}$ be a ball and $\tilde{B}_{r}$ denote either $B_{r}$ or $B_{2r}$. If $v\in W^{1,p}(\tilde{B}_{r},\mathbb{R}^{k})$ is almost $p$-harmonic in the sense that
\begin{flalign}\label{136}
\mint_{B_{r}}p\snr{Dv}^{p-1}\frac{Dv}{\snr{Dv}}\cdot D \varphi \ dx\le \tilde{\delta}\mint_{\tilde{B}_{r}}\left(\snr{Dv}^{p}+\nr{D\varphi}_{L^{\infty}(B_{r})}^{p}\right) \ dx,
\end{flalign}
for all $\varphi \in C^{\infty}_{c}(B_{r},\mathbb{R}^{k})$, then the unique map $\tilde{h} \in W^{1,p}(B_{r},\mathbb{R}^{k})$, solution to the Dirichlet Problem
\begin{flalign}\label{har1}
v+W^{1,p}_{0}(B_{r},\mathbb{R}^{k})\ni w\mapsto \min\int_{B_{r}}\snr{Dw}^{p} \ dx
\end{flalign}
satisfies
\begin{flalign}\label{har2}
\left(\mint_{B_{r}}\snr{Dv-D\tilde{h}}^{p\tilde{d}} \ dx\right)^{\frac{1}{\tilde{d}}}\le \tilde{\theta} \mint_{\tilde{B}_{r}}\snr{Dv}^{p} \ dx.
\end{flalign}
\end{lemma}
The next are a couple of simple inequalities which will be used several times throughout the paper. They are elementary, see e.g.: \cite{cosmin, RT1, tac}.
\begin{lemma}\label{L0}
The following inequalities hold true.
\begin{itemize}
\item[i.] For any $\varepsilon_{0}>0$, there exists a constant $c=c(\varepsilon_{0})$ such that for all $t\ge 0$, $l\ge m\ge 1$ there holds
$
\snr{t^{l}-t^{m}}\le c(l-m)\left(1+t^{(1+\varepsilon_{0})l}\right).
$
\item[ii.] For $t \in (0,1]$, consider the function $g(t):=t^{-ct^{\gamma}}$, where $c>0$ is an absolute constant and $\gamma \in (0,1]$. Then $\lim_{t\to 0}g(t)=1$ and $\sup_{t \in (0,1]}g(t)\le \exp(c/\gamma)$.
\end{itemize}
\end{lemma}
We conclude this section by recalling another fundamental tool in regularity theory, which will help establishing the behavior of certain quantities.
\begin{lemma}\cite{giagiu1}\label{iter}
Let $h\colon [\varrho, R_{0}]\to \mathbb{R}$ be a non-negative, bounded function and $0<\theta<1$, $0\le A$, $0<\beta$. Assume that
$
h(r)\le A(d-r)^{-\beta}+\theta h(d),
$
for $\varrho \le r<d\le R_{0}$. Then
$
h(\varrho)\le cA/(R_{0}-\varrho)^{-\beta},
$ holds, 
where $c=c(\theta, \beta)>0$.
\end{lemma}

\subsection{Some extension results}
We report some results concerning locally Lipschitz retractions. They have been extensively used in the realm of functionals with $p$-growth, see \cite{harkinlin, harlin, hop}. For integrands exihibiting $(p,q)$-growth they were used for the first time in \cite{me}, to prove that if the Lavrentiev phenomenon does not occur in the unconstrained case, then it is absent also in presence of a geometric constraint. According to assumptions $(\m1)$-$(\m2)$, $\m\subset \RN$ is a compact, $m$-dimensional $C^{3}$ Riemannian submanifold, $\partial \m =\emptyset$ and, in particular, $\m$ is $[\gamma_{2}]-1$ connected. Let us clarify this concept.
\begin{definition}\cite{hop}
Given an integer $j\ge 0$, a manifold $\m$ is said to be $j$-connected if its first $j$ homotopy groups vanish identically, that is $\pi_{0}(\m)=\pi_{1}(\m)=\cdots=\pi_{j-1}(\m)=\pi_{j}(\m)=0$.
\end{definition}
It is reasonable to expect some good properties in terms of retractions for this kind of manifolds endowed with a relatively simple topology, as the following lemma shows.
\begin{lemma}\label{hoplem}
Let $\mathcal{M}\subset \RN$, $N\ge 3$ be a compact, $j$-connected submanifold for some integer $j \in \{1,\cdots, N-2\}$ contained in an $N$-dimensional cube $Q$. Then there exists a closed $(N-j-2)$-dimensional Lipschitz polyhedron $X\subset Q\setminus \mathcal{M}$ and a locally Lipschitz retraction $\psi\colon Q\setminus X \to \mathcal{M}$ such that for any $x \in Q\setminus X$,
$
\snr{D\psi (x)}\le c/\dist(x,X)
$ holds, 
for some positive $c=c(N,j,\mathcal{M})$.
\end{lemma}
\begin{proof}
See e.g., \cite[Lemma 6.1]{harlin} for the original proof, or \cite[Lemma 4.5]{hop} for a simplified version relying on some Lipschitz extension properties of maps between Riemannian manifolds.
\end{proof}
A major technical obstruction one can face when dealing with manifold-constrained variational problems is finding comparison maps which satisfy the constraint (notice that, without further regularity details on minimizers, we cannot localize in the image). Precisely, we are no longer allowed to use convex combinations of a minimizer with a suitable cutoff function as to realize valid competitors for the problem. Hence, given any $w \in W^{1,p(\cdot)}_{\mathrm{loc}}(\Omega, \RN)$, we overcome this issue by applying Lemma \ref{hoplem} to assure a local control on the $L^{p(\cdot)}$-norm of the gradient of a suitable projected image of $w$ in terms of the $L^{p(\cdot)}$-norm of $w$ itself. This is the content of the next lemma.
\begin{lemma}[Finite energy extension.]\label{ext}
Let $\mathcal{M}$ be as in $(\mathrm{\mathcal{M}1})$-$(\mathrm{\mathcal{M}2})$ and $\U\Subset \Omega$ an open subset of $\Omega$ with Lipschitz boundary. Given $w \in W^{1,p(\cdot)}_{\mathrm{loc}}(\Omega,\RN)\cap L^{\infty}_{\mathrm{loc}}(\Omega, \RN)$ with $w(\partial \U)\subset \mathcal{M}$, there exists $\tilde{w}\in W^{1,p(\cdot)}_{w}(\U,\mathcal{M})$ satisfying
$
\int_{\U}\snr{D\tilde{w}}^{p(x)} \ dx\le c\int_{\U}\snr{Dw}^{p(x)} \ dx,
$
where $c=c(N,\mathcal{M},\gamma_{2})$.
\end{lemma}
\begin{proof}
Following \cite[Section 2.2]{hop}, we define $\mbox{Unp}(\m)$ as the set of all $x \in \RN$ for which there exists a unique point of $\m$ nearest to $x$ and, for $a\in \m$, we denote by $\mbox{reach}(\m,a)$ the supremum of the set of all numbers $r>0$ for which $\{x\in \RN\colon \snr{x-a}<r\}\subset \mbox{Unp}(A)$. Then, we can set $\mbox{reach}(\m):=\inf_{a\in \m}\mbox{Reach}(\m,a)$. Notice that, by assumptions $(\mathrm{\mathcal{M}1})$-$(\mathrm{\mathcal{M}2})$, $\mbox{reach}(\m)>0$, see \cite{fed,harlin,hop}. Now, if for some $0<\sigma <\reach(\mathcal{M})$, $V:=\left\{a \in \RN \colon \dist(a,\mathcal{M})<\sigma\right\}$ is a neighbourhood with the nearest point property, then the metric projection $\Pi\colon \overline{V}\to \mathcal{M}$ associating to any $a\in \overline{V}$ the unique $a_{0}\in \mathcal{M}$ such that $\dist(a,\m)=\snr{a-a_{0}}$, is Lipschitz continuous and $\overline{V}$ and $\mathcal{M}$ are homotopy equivalent spaces with $\pi_{i}(\overline{V})=\pi_{i}(\mathcal{M})$ for all $i \in \{0,\cdots, [\gamma_{2}]-1\}$, see e.g.: \cite[Proposition 1.17]{hat} for more details on this matter. Since $\mathcal{M}$ is compact and $w$ is bounded, there exists an $N$-dimensional cube $Q$ such that $\mathcal{M}\subset \overline{V}\subset Q$ and $\dist(w,\mathcal{M})\le \frac{1}{2}\dist (\mathcal{M},\partial Q)$ almost everywhere. By Lemma \ref{hoplem} with $j\equiv[\gamma_{2}]-1$, there exists a locally Lipschitz retraction $\psi\colon Q\setminus X \to \overline{V}$ for some $(N-[\gamma_{2}]-1)$-dimensional Lipschitz polyhedron $X\subset Q\setminus \overline{V}$, which, by construction stands strictly away from $\mathcal{M}$. Thus we have a map $P:=\Pi \circ \psi\colon Q\setminus X\to \mathcal{M}$,
satisfying 
\begin{flalign}\label{0}
\snr{\nabla P(a)}\le \frac{c}{\dist(a,X)},
\end{flalign}
for $c=c(N,\m)$. By a change of variables, the definition of the dual skeleton, the fact that $\mathcal{M}$ is $([\gamma_{2}]-1)$-connected and $\dim(X)\le N-[\gamma_{2}]-1$, there holds:
\begin{flalign}\label{1}
\int_{Q}\frac{1}{\dist(a,X)^{p(x)}} \ da\le \int_{Q}\left(1+\frac{1}{\dist(a,X)^{\gamma_{2}}}\right) \ da<c,
\end{flalign}
for a finite, positive constant $c=c(N,\mathcal{M},\gamma_{2})$. Now, for a sufficiently small $0<\varrho<\min\left\{\frac{\sigma}{2},\frac{\dist(\mathcal{M},\partial Q)}{2}\right\}$ and a point $a \in B^{N}_{\varrho}$, denote the translations $Q_{a}:=\left\{ b+a\colon \ b\in Q \right\}$ and $X_{a}:=\left\{b+a\colon \ b \in X\right\}$, so that one can define the retraction $P_{a}\colon Q_{a}\setminus X_{a}\to \mathcal{M}$ given by $P_{a}(b):=P(b-a)$. Then, by the chain rule, Fubini's theorem, \eqref{0} and \eqref{1} we obtain
\begin{flalign}\label{3}
\int_{B_{\varrho}^{N}}\int_{\U}\snr{D(P_{a}(w))}^{p(x)} \ dxda\le &\int_{\U}\snr{Dw}^{p(x)}\left(\int_{B_{\varrho}^{N}}\snr{\nabla P(w-a)}^{p(x)} \ da\right) \ dx\nonumber \\
\le &\int_{\U}\snr{Dw}^{p(x)}\left(\int_{Q}\snr{\nabla P(b)}^{p(x)} \ db\right) \ dx\le c\int_{\U}\snr{Dw}^{p(x)} \ dx,
\end{flalign}
where $c=c(N,\mathcal{M},\gamma_{2})$. Estimate \eqref{3} and Markov's inequality then render the existence of a positive $c=c(N,\mathcal{M},\gamma_{2})$ and a $\tilde{a}\in B^{N}_{\varrho}$ so that
\begin{flalign}\label{4}
\int_{\U}\snr{D(P_{\tilde{a}}(w))}^{p(x)} \ dx \le c\int_{\U}\snr{Dw}^{p(x)} \ dx,
\end{flalign}
where again $c=c(N,\mathcal{M},\gamma_{2})$.
Since $w(\partial \U)\subset \mathcal{M}$, the map $\tilde{w}:=(\left.P_{\tilde{a}}\right |_{\mathcal{M}})^{-1}\circ P_{\tilde{a}}\circ w$
is well defined and given that the inverse map $P_{\tilde{a}}^{-1}$ is Lipschitz on $\mathcal{M}$, from \eqref{4} we conclude that
$
\int_{\U}\snr{D\tilde{w}}^{p(x)} \ dx \le c\int_{\U}\snr{Dw}^{p(x)} \ dx,
$
with $c=c(N,\mathcal{M},\gamma_{2})$. Moreover, since $w(\partial U)\subset \m$, by construction we have that $\left.\tilde{w}\right |_{\partial U}=\left.w\right |_{\partial U}$ and this concludes the proof.
\end{proof}
Lemma \ref{ext} will be particularly helpful when $\U$ is a ball $B_{r}$ or an annulus $B_{r}\setminus B_{\varrho}$ for a  proper choice of $r$ and $\varrho$.

\section{Partial regularity}
In this section we first collect a couple of essential inequalities, some basic regularity results stemming only from the minimality condition and then carry out the proof of Theorem \ref{T0}.
\subsection{Basic regularity results}
The first result is Poincar\'e's inequality, well known in the unconstrained case, see \cite[Theorem 3.1]{ele}, and since it is valid for any map $w \in W^{1,p(\cdot)}_{\mathrm{loc}}(\Omega, \RN)$, it transfers verbatim for functions in $W^{1,p(\cdot)}_{\mathrm{loc}}(\Omega, \mathcal{M})$. However, given that we are dealing with bounded maps ($\m$ is compact), we present a simplified proof, including also the case in which the domain is an annulus $A_{r\theta}:=B_{r}\setminus B_{r(1-\theta)}$ for some $0<\theta<1$.
\begin{lemma}[Poincar\'e's inequality]\label{poi}
Let $w \in L^{\infty}_{\mathrm{loc}}(\Omega,\RN)\cap W^{1,p(\cdot)}_{\mathrm{loc}}(\Omega,\RN)$, with $p(\cdot)$ satisfying $\mbox{(P1)}$-$\mbox{(P2)}$ and $B_{r}\Subset \Omega$, $0<r\le 1$. Then, there holds
\begin{flalign}\label{13}
\int_{B_{r}}\left | \frac{w-(w)_{r}}{r}\right |^{p(x)} \ dx\le c\left(\int_{B_{r}}\snr{Dw}^{p(x)} \ dx +\snr{B_{r}}\right),
\end{flalign}
with $c=c(n,N,\gamma_{1},\gamma_{2},[p]_{0,\alpha},\alpha,\nr{w}_{L^{\infty}(B_{r})})$. Furthermore, if for some $0<\theta<1$, $w \in L^{\infty}(A_{r\theta},\RN)\cap W^{1,p(\cdot)}(A_{r\theta},\RN)$ is such that $\left.w\right |_{\partial B_{r}}=0$, then
\begin{flalign}\label{14}
\int_{A_{r\theta}}\snr{w/(r\theta)}^{p(x)} \ dx \le c\left(\int_{A_{r\theta}}\snr{Dw}^{p(x)} \ dx+\snr{A_{\theta}}\right),
\end{flalign}
for $c=c(n,N,\gamma_{1},\gamma_{2},[p]_{0,\alpha},\alpha,\nr{w}_{L^{\infty}(A_{r\theta})})$.
\end{lemma}
\begin{proof}
Fix $B_{r}\Subset \Omega$, $0<r\le 1$. From assumptions $\mbox{(P1)}$-$\mbox{(P2)}$, Lemma \ref{L0} (\textit{ii.}), \eqref{p1p2} and the standard Poincar\'e's inequality holding for $p\equiv p_{1}(r)$ we obtain
\begin{flalign*}
\int_{B_{r}}\left|\frac{w-(w)_{r}}{r}\right |^{p(x)} \ dx\le &cr^{p_{1}(r)-p_{2}(r)}\max\left\{1,2\nr{w}_{L^{\infty}(B_{r})}\right\}^{\gamma_{2}-\gamma_{1}}\int_{B_{r}}\left |\frac{w-(w)_{r}}{r} \right |^{p_{1}(r)} \ dx\\
\le &c\int_{B_{r}}\snr{Dw}^{p_{1}(r)} \ dx\le c\int_{B_{r}}(\snr{Du}^{p(x)}+1) \ dx,
\end{flalign*}
with $c=c(n,N,\gamma_{1},\gamma_{2},[p]_{0,\alpha},\alpha,\nr{w}_{L^{\infty}(B_{r})})$. In the same way, for $w\in L^{\infty}(A_{r\theta},\RN)\cap W^{1,p(\cdot)}(A_{r\theta},\RN)$ such that $\left.w\right|_{\partial B_{r}}=0$, we have
\begin{flalign*}
\int_{A_{r\theta}}\snr{w/(r\theta)}^{p(x)} \ dx\le c(r\theta)^{p_{1}(r\theta)-p_{2}(r\theta)}\int_{A_{r\theta}}\snr{Dw}^{p_{1}(r\theta)} \ dx\le c\int_{A_{r\theta}}(\snr{Dw}^{p(x)}+1)
 \ dx,\end{flalign*}
for $c(n,N,\gamma_{1},\gamma_{2},[p]_{0,\alpha},\alpha,\nr{w}_{L^{\infty}(A_{r\theta})})$. Here we denoted $p_{1}(r\theta):=\inf_{x\in A_{r\theta}}p(x)$ and $p_{2}(r\theta):=\sup_{x\in A_{r\theta}}p(x)$.
\end{proof}
As to successfully implement Lemma \ref{phihar}, we also need an intrinsic version of Sobolev-Poincar\'e's inequality.
\begin{lemma}[Intrinsic Sobolev-Poincar\'e's inequality]\label{inpoi}
Let $p\in (1,\infty)$ and $w \in W^{1,p}_{\mathrm{loc}}(\Omega,\RN)$. Then, there exist a positive $c=c(n,N,p)$ and exponents $d_{1}>1$ and $0<d_{2}<1$ such that
\begin{flalign*}
\left(\mint_{B_{r}}\left |\frac{w-(w)_{r}}{r} \right |^{pd_{1}} \ dx\right)^{\frac{1}{d_{1}}}\le c\left(\mint_{B_{r}}\snr{Dw}^{pd_{2}} \ dx\right)^{\frac{1}{d_{2}}}
\end{flalign*}
holds whenever $B_{r}\Subset \Omega$ is such that $0<r \le 1$. Here, $d_{1}=d_{1}(n,N,p)$ and $d_{2}=d_{2}(n,N,p)$.
\end{lemma}
\begin{proof}
We start by considering the case $1<p\le n$. Fix any $\gamma \in \left(\max\left\{\frac{1}{p},\frac{n}{n+p}\right\},1\right)$ and notice that $w \in W^{1,\gamma p}_{\mathrm{loc}}(\Omega,\RN)$ for all such $\gamma$. From the standard Sobolev-Poincar\'e's inequality we obtain
\begin{flalign*}
\left(\mint_{B_{r}}\left |\frac{w-(w)_{r}}{r} \right |^{\frac{n\gamma p}{n-\gamma p}} \ dx\right)^{\frac{n-\gamma p}{n\gamma p}}\le c\left(\mint_{B_{r}}\snr{Dw}^{\gamma p} \ dx \right)^{\frac{1}{\gamma p}},
\end{flalign*}
for $c=c(n,N,p,\gamma)$, but, being $\gamma$ ultimately influenced only by $n$ and $p$, we can conclude that $c=c(n,N,p)$. Choosing $d_{1}:=\frac{n\gamma}{n-\gamma p}>1$ since $\gamma>\frac{n}{n+p}$, and $d_{2}:=\gamma<1$ we obtain the thesis. Now, if $p>n$, then there exists $\gamma\in \left(n/p,1\right)$ so that $p\gamma>n$. Let $\kappa:=1-n/(\gamma p)$. From Morrey's embedding theorem we then have
\begin{flalign*}
\left(\mint_{B_{r}}\left |\frac{w-(w_{r})}{r} \right |^{\frac{p}{\gamma}} \ dx\right)^{\frac{\gamma}{p}}\le c[w]_{0,\kappa;B_{r}}r^{\kappa-1}\le c\left(\mint_{B_{r}}\snr{Dw}^{p\gamma} \ dx\right)^{\frac{1}{p\gamma}},
\end{flalign*}
for $c=c(n,N,p)$. Fixing $d_{1}:=\gamma^{-1}$ and $d_{2}:=\gamma$ we can conclude.
\end{proof}
\begin{remark}\label{r0}
\emph{Since $\mathcal{M}$ is compact, for a function $w$ taking values in $\mathcal{M}$ the dependence of the constants appearing in the inequalities in Lemma \ref{poi} on the $L^{\infty}$-norm of $w$ will be expressed as a dependence on $\mathcal{M}$.}
\end{remark}
In the following lemma, we present a Caccioppoli-type inequality, which is fundamental for regularity.
\begin{lemma}[Caccioppoli-type inequality]\label{cacc}
Let $u \in W^{1,p(\cdot)}(\Omega,\m)$ be a constrained local minimizers of \eqref{cvp}. Then, for any ball $B_{r}\Subset \Omega$ there holds
\begin{flalign*}
\mint_{B_{r/2}}\snr{Du}^{p(x)} \ dx \le c\mint_{B_{r}}\left |\frac{u-(u)_{r}}{r} \right|^{p(x)} \ dx,
\end{flalign*}
for $c=c(n,N,\m,\lambda,\Lambda,\gamma_{1},\gamma_{2},[p]_{0,\alpha},\alpha)$.
\end{lemma}
\begin{proof}
With $0<r/2\le s<t\le r\le 1$ we determine a cutoff function $\eta \in C^{1}_{c}(B_{r})$ such that $\chi_{B_{s}}\le \eta\le \chi_{B_{t}}$ and $\snr{D\eta}\le 2(t-s)^{-1}$ and define the map $w:=u-\eta(u-(u)_{r})$. By construction, $w \in W^{1,p(\cdot)}(B_{t},\RN)\cap L^{\infty}(B_{t},\RN)$ and $\left.w\right|_{\partial B_{t}}=\left.u\right|_{\partial B_{t}}$, so Lemma \ref{ext} renders a map $\tilde{w}\in W^{1,p(\cdot)}_{u}(B_{t},\m)$ which is an admissible competitor for $u$ in problem \eqref{cvp} and satisfies
\begin{flalign}\label{wfe}
\int_{B_{t}}\snr{D\tilde{w}}^{p(x)} \ dx \le c\int_{B_{t}}\snr{Dw}^{p(x)} \ dx \le c\left(\int_{B_{t}\setminus B_{s}}\snr{Du}^{p(x)} \ dx +\int_{B_{r}}\left |\frac{u-(u)_{r}}{t-s}\right|^{p(x)} \ dx\right),
\end{flalign}
for $c=c(N,\m,\gamma_{1},\gamma_{2})$. The minimality of $u$, $\mbox{(K2)}$, the features of $\eta$, \eqref{wfe} and \eqref{p1p2} give
\begin{flalign*}
\int_{B_{s}}\snr{Du}^{p(x)} \ dx\le &\frac{\Lambda}{\lambda}\int_{B_{t}}\snr{D\tilde{w}}^{p(x)} \ dx \\
\le &c\int_{B_{t}\setminus B_{s}}\snr{Du}^{p(x)} \ dx +c\int_{B_{r}}\left | \frac{u-(u)_{r}}{t-s}\right |^{p(x)} \ dx\\
\le& c\int_{B_{t}\setminus B_{s}}\snr{Du}^{p(x)} \ dx +c(t-s)^{-p_{2}(r)}\int_{B_{r}}\snr{u-(u)_{r}}^{p(x)} \ dx,
\end{flalign*}
with $c=c(N,\m,\lambda,\Lambda,\gamma_{1},\gamma_{2})$. Now we are in position to apply Widman's hole filling technique and Lemma \ref{iter} to conclude that
\begin{flalign*}
\int_{B_{r/2}}\snr{Du}^{p(x)} \ dx \le &cr^{-p_{2}(r)}\int_{B_{r}}\snr{u-(u)_{r}}^{p(x)} \ dx\\
&=cr^{p_{1}(r)-p_{2}(r)}\int_{B_{r}}r^{-p_{1}(r)}\snr{u-(u)_{r}}^{p(x)} \ dx\le c\int_{B_{r}}\left |\frac{u-(u)_{r}}{r} \right |^{p(x)} \ dx,
\end{flalign*}
for $c=c(n,N,\m,\lambda,\Lambda,\gamma_{1},\gamma_{2},[p]_{0,\alpha},\alpha)$. Here we also used assumption $\mbox{(P1)}$, definition \eqref{p1p2} and Lemma \ref{L0} (\textit{ii}.).
\end{proof}
The next step consists in proving an interior higher integrability result for local minimizers of \eqref{cvp}.
\begin{lemma}\label{inngeh}
Let $u\in W^{1,p(\cdot)}(\Omega, \mathcal{M})$ be a constrained local minimizer of \eqref{cvp}. Then there esists a positive integrability threshold $\tilde{\delta}_{0}=\tilde{\delta}_{0}(n,N,\m,\lambda,\Lambda,\gamma_{1},\gamma_{2},[p]_{0,\alpha},\alpha)$ such that
\begin{flalign*}
\snr{Du}^{(1+\delta)p(\cdot)}\in L^{1}_{\mathrm{loc}}(\Omega)\quad \mbox{for \ all \ }\delta \in [0,\tilde{\delta}_{0})
\end{flalign*}
and, for any $B_{r}\Subset \Omega$
\begin{flalign*}
\left(\mint_{B_{r/2}}\snr{Du}^{(1+\delta)p(x)} \ dx\right)^{\frac{1}{1+\delta}}\le c\mint_{B_{r}}(1+\snr{Du}^{2})^{p(x)/2} \ dx\quad \mbox{for \ all \ }\delta \in [0,\tilde{\delta_{0}}),
\end{flalign*}
with $c=c(n,N,\mathcal{M},\lambda,\Lambda,\gamma_{1},\gamma_{2},[p]_{0,\alpha},\alpha)$.
\end{lemma}
\begin{proof}
For a fixed $B_{r}\Subset \Omega$, combining Lemmas \ref{cacc} and \ref{inpoi} with $p\equiv p_{1}(r)$, we end up with
\begin{flalign*}
\mint_{B_{r/2}}\snr{Du}^{p(x)} \ dx \le &c\mint_{B_{r}}\left |\frac{u-(u)_{r}}{r} \right |^{p(x)} \ dx\\
\le &c\max\left\{1,2\nr{u}_{L^{\infty}(B_{r})}\right\}^{\gamma_{2}-\gamma_{1}}r^{p_{1}(r)-p_{2}(r)}\mint_{B_{r}}\left |\frac{u-(u)_{r}}{r} \right |^{p_{1}(r)} \ dx\\
\le &c\left(\mint_{B_{r}}\left |\frac{u-(u)_{r}}{r} \right |^{p_{1}(r)d_{1}} \ dx\right)^{\frac{1}{d_{1}}}\le c\left(\mint_{B_{r}}\snr{Du}^{p_{1}(r)d_{2}} \ dx\right)^{\frac{1}{d_{2}}}\le c\left(\mint_{B_{r}}\snr{Du}^{p(x)d_{2}} \ dx\right)^{\frac{1}{d_{2}}}+c,
\end{flalign*}
for $c=c(n,N,\m,\lambda,\Lambda,\gamma_{1},\gamma_{2},[p]_{0,\alpha},\alpha)$. Here we also used $\mbox{(P1)}$-$\mbox{(P2)}$, Lemma \ref{L0} (\textit{ii.}) and H\"older's inequality. Now, an application of Gehring-Giaquinta-Modica's lemma, \cite[Chapter 6]{giu} renders the existence of a positive $\tilde{\delta}_{0}=\tilde{\delta}_{0}(n,N,\m,\lambda,\Lambda,\gamma_{1},\gamma_{2},[p]_{0,\alpha},\alpha)$ so that
\begin{flalign*}
\left(\mint_{B_{r/2}}\snr{Du}^{(1+\delta)p(x)} \ dx\right)^{\frac{1}{1+\delta}}\le c\mint_{B_{r}}(1+\snr{Du}^{2})^{p(x)/2} \ dx,
\end{flalign*}
with $c=c(n,N,\m,\lambda,\Lambda,\gamma_{1},\gamma_{2},[p]_{0,\alpha},\alpha)$, for all $\delta \in [0,\tilde{\delta}_{0})$. Finally, after a standard covering argument, we obtain that $\snr{Du}^{(1+\delta)p(\cdot)}\in L^{1}_{\mathrm{loc}}(\Omega)$ for all $\delta \in [0,\tilde{\delta}_{0})$.
\end{proof}
\begin{remark}\em{Before proceding further we need to stress that, if $B_{r}\Subset \Omega$ and $w\in W^{1,p}(B_{r},\RN)$ is such that $w\equiv 0$ on $U\subset B_{r}$ with $\snr{U}>\tilde{c}\snr{B_{r}}$ for some positive, absolute $\tilde{c}$, then Sobolev-Poincar\'e's inequality gives
\begin{flalign}\label{poi0}
\int_{B_{r}}\snr{w/r}^{p} \ dx\le cr^{-n(p/p_{*}-1)}\left(\int_{B_{r}}\snr{Dw}^{p_{*}} \ dx\right)^{\frac{p}{p_{*}}},
\end{flalign}
for $c=c(n,N,p,\tilde{c})$. Here $p_{*}:=\max\left\{1,\frac{np}{n+p}\right\}$, as usual.
}

\end{remark}
The following lemma is an up to the boundary higher integrability result. The argument is well-known to specialists, see \cite{acefus, RT1}, and it essentially relies on the fact that Caccioppoli's inequality can be carried up to the boundary. However, we did not manage to find in the literature a proof for the manifold-constrained case, so we shall report it here.
\begin{lemma}\label{bougeh}
Let $p \in [\gamma_{1},\gamma_{2}]$, $u \in W^{1,p}_{\mathrm{loc}}(\Omega,\m)$ be such that $\snr{Du}^{p(1+\delta_{1})}\in L^{1}_{\mathrm{loc}}(\Omega)$ for some $\delta_{1}>0$ and let $v\in W^{1,p}_{u}(B_{r},\m)$ be a solution to the Dirichlet problem
\begin{flalign*}
W^{1,p}_{u}(B_{r},\m)\ni w\mapsto \min \int_{B_{r}}k_{0}\snr{Dw}^{p} \ dx,
\end{flalign*}
where $k_{0}\in [\lambda,\Lambda]$ is a positive constant and $B_{r}\Subset \Omega$ is any ball with $r \in (0,1]$. Then there exists a positive integrability threshold $\tilde{\sigma}_{0}\in (0,\delta_{1})$ such that
\begin{flalign*}
\mint_{B_{r}}(1+\snr{Dv}^{2})^{(1+\sigma)p/2} \ dx\le c\mint_{B_{r}}(1+\snr{Du}^{2})^{(1+\sigma)p/2} \ dx\quad \mbox{for \ all \ }\sigma \in [0,\tilde{\sigma}_{0}).
\end{flalign*}
Here $\tilde{\sigma}_{0}=\tilde{\sigma}_{0}(n,N,\m,\lambda,\Lambda,\gamma_{1},\gamma_{2})$ and $c=c(n,N,\m,\lambda,\Lambda,\gamma_{1},\gamma_{2})$.
\end{lemma}
\begin{proof}
With $x_{0}\in B_{r}$, let us fix a ball $B_{\varrho}(x_{0})\subset \mathbb{R}^{n}$, $0<\varrho\le 1$. We start with the case in which $\snr{B_{\varrho}(x_{0})\setminus B_{r}}>\snr{B_{\varrho}(x_{0})}/10$. Let us fix parameters $0<\varrho/2\le s<t\le \varrho$ and consider $\eta \in C^{1}_{c}(B_{t}(x_{0}))$ such that $\chi_{B_{s}(x_{0})}\le \eta \le \chi_{B_{t}(x_{0})}$ and $\snr{D\eta}\le 2(t-s)^{-1}$. The function $w:=v-\eta(v-u)$ coincides with $v$ on $\partial B_{r}$ and on $\partial (B_{r}\cap B_{t}(x_{0}))$ in the sense of traces, so Lemma \ref{ext} with $p(\cdot)$ equal to constant $p$, provides us with a map $\tilde{w}\in W^{1,p}_{v}(B_{r}\cap B_{t}(x_{0}),\m)$ such that
\begin{flalign}\label{bg0}
\int_{B_{r}\cap B_{t}(x_{0})}\snr{D\tilde{w}}^{p} \ dx \le& c\int_{B_{r}\cap B_{t}(x_{0})}\snr{Dw}^{p} \ dx\nonumber \\
\le& c\int_{B_{r}\cap (B_{t}(x_{0})\setminus B_{s}(x_{0}))}\snr{Dv}^{p} \ dx+c\int_{B_{r}\cap B_{\varrho}(x_{0})}\snr{Du}^{p}+\left |\frac{v-u}{t-s} \right |^{p} \ dx,
\end{flalign}
for $c=c(N,\m,\gamma_{1},\gamma_{2})$. The minimality of $v$ in the Dirichlet class $W^{1,p}_{v}(B_{r}\cap B_{t}(x_{0}),\m)$ and \eqref{bg0} render
\begin{flalign*}
\int_{B_{r}\cap B_{s}(x_{0})}\snr{Dv}^{p} \ dx \le&\frac{\Lambda}{\lambda}\int_{B_{r}\cap B_{t}(x_{0})}\snr{D\tilde{w}}^{p} \ dx\\
\le &c\int_{B_{r}\cap (B_{t}(x_{0})\setminus B_{s}(x_{0}))}\snr{Dv}^{p} \ dx+c\int_{B_{r}\cap B_{\varrho}(x_{0})}\snr{Du}^{p}+\left |\frac{v-u}{t-s} \right |^{p} \ dx,
\end{flalign*}
for $c=c(N,\m,\lambda,\Lambda,\gamma_{1},\gamma_{2})$. By filling the hole and applying Lemma \ref{iter}, we get
\begin{flalign}\label{bg1}
\int_{B_{r}\cap B_{\varrho/2}(x_{0})}\snr{Dv}^{p} \ dx \le c\int_{B_{r}\cap B_{\varrho}(x_{0})}\snr{Du}^{p}+\left |\frac{u-v}{\varrho} \right |^{p} \ dx,
\end{flalign}
with $c=c(n,N,\m,\lambda,\Lambda,\gamma_{1},\gamma_{2})$. Notice that we can extend $v-u\equiv 0$ outside $B_{r}$, since $u-v \in W^{1,p}_{0}(B_{r},\RN)$, so there are no discontinuities on $\partial (B_{r}\cap B_{\varrho}(x_{0}))$. Recalling also that $\snr{B_{\varrho}(x_{0})\setminus B_{r}}>\snr{B_{\varrho}(x_{0})}/10$, from \eqref{poi0} we have that
\begin{flalign*}
\int_{B_{r}\cap B_{\varrho}(x_{0})}&\left |\frac{u-v}{\varrho} \right |^{p} \ dx=\int_{B_{\varrho}(x_{0})}\left |\frac{u-v}{\varrho} \right |^{p} \ dx\nonumber \\
\le&c\varrho^{-n(p/p_{*}-1)}\left(\int_{B_{\varrho}(x_{0})}\snr{Du-Dv}^{p_{*}} \ dx\right)^{\frac{p}{p_{*}}}=c\varrho^{-n(p/p_{*}-1)}\left(\int_{B_{r}\cap B_{\varrho}(x_{0})}\snr{Du-Dv}^{p_{*}} \ dx\right)^{\frac{p}{p_{*}}},
\end{flalign*}
for $c=c(n,N,\gamma_{1},\gamma_{2})$. Averaging in the previous display and keeping in mind that $\snr{B_{\rr}(x_{0})\cap B_{r}}\le \snr{B_{\rr}(x_{0})}$, we obtain
\begin{flalign}\label{bg2}
\mint_{B_{r}\cap B_{\varrho}(x_{0})}&\left |\frac{u-v}{\varrho} \right |^{p} \ dx\le c\left(\mint_{B_{r}\cap B_{\varrho}(x_{0})}\snr{Du-Dv}^{p_{*}} \ dx\right)^{\frac{p}{p_{*}}},
\end{flalign}
with $c=c(n,N,\m,\lambda,\Lambda,\gamma_{1},\gamma_{2})$. Coupling \eqref{bg1} and \eqref{bg2} we get, by triangle and H\"older's inequalities,
\begin{flalign*}
\mint_{B_{r}\cap B_{\varrho/2}(x_{0})}\snr{Dv}^{p} \ dx \le c\left\{\mint_{B_{r}\cap B_{\varrho}(x_{0})}\snr{Du}^{p} \ dx +\left(\mint_{B_{r}\cap B_{\varrho}(x_{0})}\snr{Dv}^{p_{*}} \ dx\right)^{\frac{p}{p_{*}}}\right\},
\end{flalign*}
with $c=c(n,N,\m,\lambda,\Lambda,\gamma_{1},\gamma_{2})$. We next consider the situation when it is $B_{\varrho}(x_{0})\Subset B_{r}$. In this case, we can apply the standard Sobolev-Poincar\'e's inequality, thus getting, as in the interior case, (Lemma \ref{inngeh} with $p(\cdot)$ equal to constant $p$),
\begin{flalign*}
\mint_{B_{\varrho/2}(x_{0})}\snr{Dv}^{p} \ dx \le c\left(\mint_{B_{\varrho}(x_{0})}\snr{Dv}^{p_{*}} \ dx\right)^{\frac{p}{p_{*}}},
\end{flalign*}
for $c=c(n,N,\m,\lambda,\Lambda,\gamma_{1},\gamma_{2})$. The two cases can be combined via a standard covering argument. Precisely, upon defining
\begin{flalign*}
V(x):=\begin{cases}
\ \snr{Dv(x)}^{p_{*}} \ \ & x \in B_{\varrho}(x_{0})\\
\ 0 \ \ &x\in \mathbb{R}^{n}\setminus B_{\varrho}(x_{0})
\end{cases}\quad \mbox{and}\quad U(x):=\begin{cases} \ \snr{Du(x)}^{p} \ \ & x\in B_{\varrho}(x_{0})\\
\ 0 \ \ & x\in \mathbb{R}^{n}\setminus B_{\varrho}(x_{0})
\end{cases},
\end{flalign*} 
we get 
\begin{flalign*}
\mint_{B_{\varrho/2}(x_{0})}V(x)^{\frac{p}{p_{*}}} \ dx \le c\left\{\mint_{B_{\varrho}(x_{0})}U(x) \ dx +\left(\mint_{B_{\varrho}(x_{0})}V(x) \ dx\right)^{\frac{p}{p_{*}}}\right\},
\end{flalign*}
with $c=c(n,N,\m,\lambda,\Lambda,\gamma_{1},\gamma_{2})$. At this point, by a variant of Gehring's Lemma, we obtain that there exists a positive $\tilde{\sigma}_{0}=\tilde{\sigma}_{0}(n,N,\m,\lambda,\Lambda,\gamma_{1},\gamma_{2})$ such that $0<\tilde{\sigma}<\delta_{1}$ and
\begin{flalign}\label{rev0}
\left(\mint_{B_{r}}\snr{Dv}^{p(1+\sigma)} \ dx\right)^{\frac{1}{1+\sigma}}\le c\left\{\mint_{B_{r}}\snr{Dv}^{p} \ dx +\left(\mint_{B_{r}}\snr{Du}^{(1+\sigma)p} \ dx\right)^{\frac{1}{1+\sigma}}\right\},
\end{flalign}
for all $\sigma\in[0,\tilde{\sigma}_{0})$ where $c=c(n,N,\m,\lambda,\Lambda,\gamma_{1},\gamma_{2})$, see \cite[Theorem 3 and Proposition 1, Chapter 2]{giamodsou}. From \eqref{rev0} and the minimality of $v$ within the Dirichlet class $W^{1,p}_{u}(B_{r},\m)$ we can conclude that
\begin{flalign*}
\mint_{B_{r}}\snr{Dv}^{p(1+\sigma)} \ dx \le c\mint_{B_{r}}\snr{Du}^{p(1+\sigma)} \ dx \ \ \mathrm{for \ all} \ \ \sigma \in [0,\tilde{\sigma}_{0}),
\end{flalign*}
with $c=c(n,N,\m,\lambda,\Lambda,\gamma_{1},\gamma_{2})$.
\end{proof}
The next corollary allows recovering some useful estimates for the average of the gradient of solutions to problem \eqref{cvp}.
\begin{corollary}\label{C0}
Let $u \in W^{1,p(\cdot)}(\Omega,\m)$ be a constrained local minimizer of \eqref{cvp}. Then, for any $B_{r}\subset \Omega$ with $r\in (0,1]$, such that $B_{4r}\Subset \Omega$ there holds
\begin{flalign}\label{16}
&\mint_{B_{r}}\snr{Du}^{p(x)} \ dx \le cr^{-p_{2}(2r)},\\
&\mint_{B_{r}}\snr{Du}^{p(x)(1+\delta)} \ dx \le cr^{-p_{2}(4r)(1+\delta)},\label{17}
\end{flalign}
where $c=c(n,N,\m,\lambda,\Lambda,\gamma_{1},\gamma_{2},[p]_{0,\alpha},\alpha)$ and $\delta\in (0,\tilde{\delta}_{0})$, where $\tilde{\delta}_{0}$ is the higher integrability threshold given by Lemma \ref{inngeh}.
\end{corollary}
\begin{proof}
Inequality \eqref{16} comes from an application of Lemma \ref{cacc} and the boundedness of $u$. In fact we have
\begin{flalign*}
\mint_{B_{r}}\snr{Du}^{p(x)} \ dx \le c\mint_{B_{2r}}\left |\frac{u-(u)_{2r}}{r} \right |^{p(x)} \ dx\le c\max\left\{1,2\nr{u}_{L^{\infty}(B_{2r})}\right\}^{\gamma_{2}}r^{-p_{2}(2r)},
\end{flalign*}
for $c=c(n,N,\m,\lambda,\Lambda,\gamma_{1},\gamma_{2},[p]_{0,\alpha},\alpha)$. On the other hand, combining Lemmas \ref{inngeh} and \ref{cacc}, we have
\begin{flalign*}
\mint_{B_{r}}\snr{Du}^{(1+\delta)p(x)} \ dx \le c\left(\mint_{B_{2r}}(1+\snr{Du}^{2})^{p(x)/2} \ dx\right)^{1+\delta}\le c\left(1+\mint_{B_{4r}}\left | \frac{u-(u)_{4r}}{r}\right |^{p(x)} \ dx\right)^{1+\delta}\le cr^{-p_{2}(r)(1+\delta)},
\end{flalign*}
with $c=c(n,N,\m,\lambda,\Lambda,\gamma_{1},\gamma_{2},[p]_{0,\alpha},\alpha)$, for any $\delta \in (0,\tilde{\delta}_{0})$.
\end{proof}
\subsection{Proof of Theorem \ref{T0}}
Now we are ready to prove Theorem \ref{T0}. For the reader's convenience, we shall split the proof in seven steps.\\\\
\emph{Step 1: comparison, first time.} We define $\delta_{0}:=\frac{1}{2}\min\left\{\tilde{\delta}_{0},1\right\}$, where $\tilde{\delta}_{0}$ is the higher integrability threshold from Lemma \ref{inngeh}. Notice that by $\mbox{(P1)}$, the set
\begin{flalign}\label{om1}
\Omega^{+}:=\left\{x \in \Omega\colon p(x)> n-\frac{\delta_{0}}{2}\right\},
\end{flalign}
is open and by Lemma \ref{inngeh}, $\snr{Du}^{(1+\delta_{0})p(\cdot)}\in L^{1}(\Omega^{+})$, thus $u \in W^{1,n+\frac{\delta_{0}}{4}}(\Omega^{+},\m)$. An application of Morrey's embedding theorem then renders that $u \in C^{0,\beta'}(\Omega^{+},\m)$ with $\beta':=\frac{\delta_{0}}{4n+\delta_{0}}$. We will treat this case in \emph{Step 7}, so, from now on, $\gamma_{2}<n$ holds. We set
\begin{flalign}\label{sigma0}
\sigma_{0}:=\frac{1}{2}\min\left\{\delta_{0},\frac{\alpha}{2\max\{\gamma_{2},n\}}\right\}
\end{flalign}
and fix a $\tilde{R}_{*}\in (0,1]$ so small that
\begin{flalign}\label{18}
[p]_{0,\alpha}4\tilde{R}_{*}^{\alpha}\le \frac{\sigma_{0}\gamma_{1}}{\sigma_{0}+2}
\end{flalign}
is satisfied on $\bar{B}_{\tilde{R}_{*}}\equiv \bar{B}_{\tilde{R}_{*}}(x_{0})\Subset \Omega$. Clearly this condition transfers on any ball $B_{r}(x_{1})\subset B_{\tilde{R}_{*}}$. We select also an $R_{*}\in (0,\tilde{R}_{*}/2)$, whose size will be specified along the proof. Now notice that, since Lemma \ref{inngeh} holds true for all balls $B_{4r}\subset B_{R_{*}}\subset B_{\tilde{R}_{*}}$, $\snr{Du}^{(1+\delta)p_{1}(2r)}\in L^{1}(B_{2r})$ for all $\delta \in (0,\delta_{0}]$. Therefore, by \eqref{18} and assumption $\mbox{(P1)}$ it easily follows that
\begin{flalign}\label{kom}
p_{2}(2r)<\left(1+\frac{\sigma_{0}}{2}\right)p_{2}(2r)\le (1+\sigma_{0})p_{1}(2r),
\end{flalign}
so, recalling that $\sigma_{0}<\delta_{0}$, we get
\begin{flalign}\label{19}
\snr{Du}^{(1+\sigma/2)p_{2}(2r)}\in L^{1}(B_{2r}) \quad \mbox{for \ all \ }\sigma \in (0,\sigma_{0}].
\end{flalign}
On such a ball we impose the following smallness condition on the energy: there exists an $\varepsilon\in (0,1)$, whose size will be fixed later on, such that
\begin{flalign}\label{small}
\left((2r)^{p_{2}(2r)-n}\int_{B_{2r}}(1+\snr{Du}^{2})^{p_{2}(2r)/2} \ dx \right)^{\frac{1}{p_{2}(2r)}}<\varepsilon.
\end{flalign}
Let $v\in W^{1,p_{2}(2r)}_{u}(B_{r},\m)$ be a solution to the frozen Dirichlet problem
\begin{flalign}\label{fdp}
\inf_{w\in W^{1,p_{2}(2r)}_{u}(B_{r},\m)}\mathcal{G}(w,B_{r}):=\inf_{w\in W^{1,p_{2}(2r)}_{u}(B_{r},\m)}\int_{B_{r}}k_{0}\snr{Dw}^{p_{2}(2r)} \ dx,
\end{flalign}
where $k_{0}:=k(x_{0})$ is the value the coefficient $k(\cdot)$ attains in the centre of $B_{r}$. Needless to say, being $\mbox{(K2)}$ in force, $k(\cdot)$ ranges between two positive, absolute constants $\lambda$ and $\Lambda$, so none of the estimates we will provide is going to depend on $x_{0}$. By minimality, $v$ solves the Euler-Lagrange equation
\begin{flalign}\label{elv}
0=\int_{B_{r}}k_{0}p_{2}(2r)\snr{Dv}^{p_{2}(2r)-2}\left(Dv\cdot D\varphi-A_{v}(Dv,Dv)\varphi\right) \ dx,
\end{flalign}
for any $\varphi \in W^{1,p_{2}(2r)}_{0}(B_{r},\RN)\cap L^{\infty}(B_{r},\RN)$, where, for $y \in \m$, $A_{y}\colon T_{y}\m\times T_{y}\m\to (T_{y}\m)^{\perp}$ denotes the second fundamental form of $\m$. In particular, by tangentiality,
\begin{flalign}\label{20}
\nabla^{2}\Pi(v)(Dv,Dv)=-A_{v}(Dv,Dv)\quad \mbox{and}\quad \snr{A_{v}(Dv,Dv)}\le c_{\m}\snr{Dv}^{2},
\end{flalign}
where $c_{\m}$ depends only on the geometry of $\m$, see \cite[Appendix to Chapter 2]{sim}. In all the forthcoming estimates, any dependency on $c_{\m}$ of the constants will always be denoted as a dependency on $\m$, i.e.: $c(c_{\m})\equiv c(\m)$. From \eqref{19}, the compactness of $\m$ and the fact that $\left.v\right |_{\partial B_{r}}=\left.u\right |_{\partial B_{r}}$, we see that the map $\varphi:=u-v$ is admissible in \eqref{elv}. Let us define
\begin{flalign}\label{sigma}
\sigma:=\frac{1}{2}\min\left\{\tilde{\sigma}_{0}, \sigma_{0}\right\},
\end{flalign}
where $\tilde{\sigma}_{0}$ is the boundary higher integrability threshold given by Lemma \ref{bougeh}. Now, exploiting assumptions $(P2)$-$(K2)$, $\eqref{20}_{2}$ and H\"older's inequality we estimate
\begin{flalign}\label{s10}
&\left | \ \int_{B_{r}}k_{0}p_{2}(2r)\snr{Dv}^{p_{2}(2r)-2}A_{v}(Dv,Dv)(u-v)  \ dx\ \right |\nonumber \\
&\quad\quad\quad\quad\quad\quad\quad\le c\int_{B_{r}}\snr{Dv}^{p_{2}(2r)}\snr{u-v} \ dx\nonumber \\
&\quad\quad\quad\quad\quad\quad\quad\le cr^{n}\left(\mint_{B_{r}}\snr{Dv}^{(1+\sigma)p_{2}(2r)} \ dx\right)^{\frac{1}{1+\sigma}}\left(\mint_{B_{r}}\snr{u-v}^{p_{2}(2r)} \ dx\right)^{\frac{\sigma}{1+\sigma}}=:cr^{n}\left[\mbox{(I)}\cdot\mbox{(II)}\right],
\end{flalign}
where $c=c(n,N,\m,\lambda,\Lambda,\gamma_{1},\gamma_{2})$. Using Lemma \ref{bougeh}, \eqref{sigma}, H\"older's inequality, \eqref{p1p2}, assumptions $\mbox{(P1)}$-$\mbox{(P2)}$, \eqref{kom}, Lemma \ref{inngeh}, \eqref{small} and Lemma \ref{L0} (\textit{ii.}) we have
\begin{flalign*}
\mbox{(I)}\le &c\left(\mint_{B_{r}}\snr{Du}^{(1+\sigma)p_{2}(2r)} \ dx\right)^{\frac{1}{1+\sigma}}\le c\left(\mint_{B_{r}}\snr{Du}^{(1+\sigma_{0})p(x)} \ dx\right)^{\frac{p_{2}(2r)}{(1+\sigma_{0})p_{1}(r)}}\nonumber \\
\le& c\left(\mint_{B_{2r}}(1+\snr{Du}^{2})^{p(x)/2} \ dx\right)^{\frac{p_{2}(2r)}{p_{1}(2r)}-1}\mint_{B_{2r}}(1+\snr{Du}^{2})^{p_{2}(2r)/2} \ dx\nonumber \\
\le &cr^{-4^{\alpha}r^{\alpha}[p]_{0,\alpha}\frac{\gamma_{2}}{\gamma_{1}}}\left((2r)^{p_{2}(2r)-n}\int_{B_{2r}}(1+\snr{Du}^{2})^{p_{2}(2r)/2} \ dx\right)^{\frac{p_{2}(2r)-p_{1}(2r)}{p_{1}(2r)}}\mint_{B_{2r}}(1+\snr{Du}^{2})^{p_{2}(2r)/2} \ dx\nonumber \\
\le &c\varepsilon^{\frac{p_{2}(2r)(p_{2}(2r)-p_{1}(2r))}{p_{1}(2r)}}\mint_{B_{2r}}(1+\snr{Du}^{2})^{p_{2}(2r)/2} \ dx\le c\mint_{B_{2r}}(1+\snr{Du}^{2})^{p_{2}(2r)/2} \ dx,
\end{flalign*}
with $c=c(n,N,\m,\lambda,\Lambda,\gamma_{1},\gamma_{2},[p]_{0,\alpha},\alpha)$. On the other hand, by Poincar\'e's inequality, the minimality of $v$ in class $W^{1,p_{2}(2r)}_{u}(B_{r},\m)$ and \eqref{small} we bound
\begin{flalign*}
\mbox{(II)}\le &c\left(r^{p_{2}(2r)}\mint_{B_{r}}\snr{Du-Dv}^{p_{2}(2r)} \ dx\right)^{\frac{\sigma}{1+\sigma}}\nonumber \\
\le& c\left((2r)^{p_{2}(2r)-n}\int_{B_{2r}}(1+\snr{Du}^{2})^{p_{2}(2r)/2} \ dx\right)^{\frac{\sigma}{1+\sigma}}\le c\varepsilon^{\frac{\gamma_{1}\sigma}{1+\sigma}},
\end{flalign*}
for $c=c(n,N,\lambda,\Lambda,\gamma_{1},\gamma_{2})$. Inserting the content of the previous two displays in \eqref{s10} we obtain
\begin{flalign}\label{21}
\left | \ \int_{B_{r}}k_{0}p_{2}(2r)\snr{Dv}^{p_{2}(2r)-2}A_{v}(Dv,Dv)(u-v)  \ dx\ \right |\le c\varepsilon^{\frac{\gamma_{1}\sigma}{1+\sigma}}\int_{B_{2r}}(1+\snr{Du}^{2})^{p_{2}(2r)/2} \ dx,
\end{flalign}
for $c=c(n,N,\m,\lambda,\Lambda,\gamma_{1},\gamma_{2},[p]_{0,\alpha},\alpha)$. For the ease of notation, if $z\in \mathbb{R}^{N\times n}$, let us name $g(z):=k_{0}\snr{z}^{p_{2}(2r)}$. The convexity of $g(\cdot)$ and \eqref{elv} then render
\begin{flalign*}
\mathcal{G}(u,B_{r})-\mathcal{G}(v,B_{r})=&\int_{B_{r}}\partial g(Dv)(Du-Dv) \ dx \\
&+\int_{B_{r}}\left(\int_{0}^{1}(1-t)\partial^{2}g(t Du+(1-t)Dv) \ dt\right)(Du-Dv)(Du-Dv) \ dx\\
\ge &-\left| \ \int_{B_{r}}k_{0}p_{2}(2r)\snr{Dv}^{p_{2}(2r)-2}A_{v}(Dv,Dv)(u-v)  \ dx \ \right |\\
&+c\int_{B_{r}}\left(\int_{0}^{1}(1-t)\snr{t Du+(1-t)Dv}^{p_{2}(2r)-2} \ dt\right)\snr{Du-Dv}^{2} \ dx,
\end{flalign*}
with $c=c(n,N,\lambda,\gamma_{1},\gamma_{2})$. From this and \eqref{21} we obtain
\begin{flalign}\label{22}
c\int_{B_{r}}(\snr{Dv}^{2}+\snr{Du}^{2})^{\frac{p_{2}(2r)-2}{2}}\snr{Du-Dv}^{2} \ dx\le&c\varepsilon^{\frac{\gamma_{1}\sigma}{1+\sigma}}\int_{B_{2r}}(1+\snr{Du}^{2})^{\frac{p_{2}(2r)}{2}} \ dx+\g(u,B_{r})-\g(v,B_{r}),
\end{flalign}
with $c=c(n, N,\m, \lambda,\Lambda,\gamma_{1},\gamma_{2}, [p]_{0,\alpha},\alpha)$. Using this time the minimality of $u$, we see that
\begin{flalign*}
\mathcal{G}(u,B_{r})-\mathcal{G}(v,B_{r})\le &\mathcal{G}(u,B_{r})-\mathcal{G}(v,B_{r})+\mathcal{E}(v,B_{r})-\mathcal{E}(u,B_{r})\\
\le &\snr{\mathcal{G}(u,B_{r})-\mathcal{E}(u,B_{r})}+\snr{\mathcal{E}(v,B_{r})-\mathcal{G}(v,B_{r})}.
\end{flalign*}
Recall the definitions of $\sigma_{0}$ and of $k_{0}$. From assumptions $\mbox{(K1)}$-$\mbox{(K2)}$ and $\mbox{(P1)}$-$\mbox{(P2)}$, Lemma \ref{L0} (\textit{i.}) with $\varepsilon_{0}\equiv\sigma_{0}/2$, \eqref{kom}, Lemma \ref{inngeh} and \eqref{small} we obtain 
\begin{flalign*}
\snr{\mathcal{G}(u,B_{r})-\mathcal{E}(u,B_{r})}\le &\int_{B_{r}}\snr{k_{0}-k(x)}\snr{Du}^{p_{2}(2r)} \ dx +\Lambda\int_{B_{r}}\left |\snr{Du}^{p_{2}(2r)}-\snr{Du}^{p(x)} \right | \ dx\\
\le &cr^{\alpha+n}\mint_{B_{r}}1+\snr{Du}^{p_{2}(2r)(1+\sigma_{0}/2)} \ dx\le cr^{\alpha+n}\mint_{B_{r}}1+\snr{Du}^{p_{1}(2r)(1+\sigma_{0})} \ dx\\
\le &cr^{\alpha+n}\left(\mint_{B_{2r}}(1+\snr{Du}^{2})^{p_{2}(2r)/2} \ dx\right)^{1+\sigma_{0}}\\
\le& cr^{\alpha+n-\gamma_{2}\sigma_{0}}\left((2r)^{p_{2}(2r)-n}\int_{B_{2r}}(1+\snr{Du}^{2})^{p_{2}(2r)/2} \ dx\right)^{\sigma_{0}}\mint_{B_{2r}}(1+\snr{Du}^{2})^{p_{2}(2r)/2} \ dx\\
\le &cr^{\kappa}\varepsilon^{\gamma_{1}\sigma_{0}}\int_{B_{2r}}(1+\snr{Du}^{2})^{p_{2}(2r)/2} \ dx,
\end{flalign*}
where $c(\texttt{data})$ and $\kappa:=\alpha-\gamma_{2}\sigma_{0}>0$ because of \eqref{sigma0}. Choosing now $\sigma$ as in \eqref{sigma}, using also Lemma \ref{bougeh} and \eqref{kom}, in a totally similar way we get
\begin{flalign*}
\snr{\mathcal{E}(v,B_{r})-\mathcal{G}(v,B_{r})}\le & cr^{\alpha+n}\mint_{B_{r}}1+\snr{Dv}^{(1+\sigma)p_{2}(2r)} \ dx\le cr^{\alpha+n}\mint_{B_{r}}1+\snr{Du}^{(1+\sigma)p_{2}(2r)} \ dx\\
\le& cr^{\alpha+n}\mint_{B_{r}}1+\snr{Du}^{(1+\sigma_{0})p(x)} \ dx\le cr^{\alpha+n}\left(\mint_{B_{2r}}(1+\snr{Du})^{p(x)/2} \ dx\right)^{1+\sigma_{0}}\\
\le& cr^{\kappa}\varepsilon^{\gamma_{1}\sigma_{0}}\int_{B_{2r}}(1+\snr{Du}^{2})^{p_{2}(2r)/2} \ dx,
\end{flalign*}
where we set $\varepsilon_{0}\equiv\sigma$ while applying Lemma \ref{L0} (\textit{i.}). Here $c(\texttt{data})$ and $\kappa>0$ is as before. All in all, remembering that, by definition, $0<\sigma<\sigma_{0}$, we can conclude
\begin{flalign}\label{24}
\int_{B_{r}}(\snr{Du}^{2}+\snr{Dv}^{2})^{\frac{p_{2}(2r)-2}{2}}\snr{Du-Dv}^{2} \ dx \le c\left[\varepsilon^{\frac{\gamma_{1}\sigma}{1+\sigma}}+r^{\kappa}\varepsilon^{\gamma_{1}\sigma}\right]\int_{B_{2r}}(1+\snr{Du}^{2})^{p_{2}(2r)/2} \ dx.
\end{flalign}
Since the next estimates will be slightly different for the cases $p_{2}(2r)\ge 2$ or $1<p_{2}(2r)<2$, we introduce the quantities
\begin{flalign*}
\kappa_{1}:=\begin{cases}
\ \kappa \quad & 2\le p_{2}(2r)\\
\ \frac{\kappa p_{2}(2r)}{2} \qquad & 1<p_{2}(2r)<2
\end{cases}, \quad \kappa_{2}:=\begin{cases}
\ \gamma_{1}\sigma \quad & 2\le p_{2}(2r)\\
\ \frac{\gamma_{1}\sigma p_{2}(2r)}{2} \quad & 1<p_{2}(2r)<2
\end{cases},\\\\
\kappa_{3}:=\begin{cases}
\ \frac{\gamma_{1}\sigma}{1+\sigma} \quad & 2\le p_{2}(2r)\\
\ \frac{\gamma_{1}p_{2}(2r)\sigma}{2(1+\sigma)}\quad & 1<p_{2}(2r)<2
\end{cases}.\quad \quad \quad \quad \quad \quad \quad \quad 
\end{flalign*}
Now, if $p_{2}(2r)\ge 2$, then we directly have
\begin{flalign}\label{36}
\int_{B_{r}}\snr{Du-Dv}^{p_{2}(2r)} \ dx \le& c\int_{B_{r}}(\snr{Du}^{2}+\snr{Dv}^{2})^{\frac{p_{2}(2r)-2}{2}}\snr{Du-Dv}^{2} \ dx \nonumber \\
\le &c\left[\varepsilon^{\frac{\gamma_{1}\sigma}{1+\sigma}}+r^{\kappa}\varepsilon^{\gamma_{1}\sigma}\right]\int_{B_{2r}}(1+\snr{Du}^{2})^{p_{2}(2r)/2} \ dx,
\end{flalign}
while, if $1<p_{2}(2r)<2$, by H\"older's inequality and the minimality of $v$ we obtain
\begin{flalign}\label{37}
\int_{B_{r}}\snr{Du-Dv}^{p_{2}(2r)} \ dx \le& \left(\int_{B_{r}}(\snr{Du}^{2}+\snr{Dv}^{2})^{\frac{p_{2}(2r)-2}{2}}\snr{Du-Dv}^{2} \ dx\right)^{\frac{p_{2}(2r)}{2}}\left(\int_{B_{r}}(\snr{Du}^{2}+\snr{Dv}^{2})^{\frac{p_{2}(2r)}{2}} \ dx\right)^{\frac{2-p_{2}(2r)}{2}}\nonumber \\
\le &c\left[\varepsilon^{\frac{\gamma_{1}\sigma p_{2}(2r)}{2(1+\sigma)}}+r^{\frac{\kappa p_{2}(2r)}{2}}\varepsilon^{\frac{\gamma_{1}\sigma p_{2}(2r)}{2}}\right]\int_{B_{2r}}(1+\snr{Du}^{2})^{p_{2}(2r)/2} \ dx.
\end{flalign}
Coupling estimates \eqref{36} and \eqref{37}, we can conclude in any case that
\begin{flalign}\label{38}
\int_{B_{r}}\snr{Du-Dv}^{p_{2}(2r)} \ dx \le c\left[\varepsilon^{\kappa_{3}}+r^{\kappa_{1}}\varepsilon^{\kappa_{2}}\right]\int_{B_{2r}}(1+\snr{Du}^{2})^{\frac{p_{2}(2r)}{2}} \ dx,
\end{flalign}
where $c=c(\texttt{data})$. As mentioned before, our choice of $\sigma_{0}$ assures the positivity of $\kappa_{1}$ as well.\\\\
\emph{Step 2: harmonic approximation.}
We aim to apply Lemma \ref{phihar} in order to obtain an unconstrained $p_{2}(2r)$-harmonic map suitably close to $v$. Hence, we need to transfer condition \eqref{small} from $u$ to $v$. From the minimality of $v$ in class $W^{1,p_{2}(2r)}_{u}(B_{r},\m)$ and \eqref{small} we see that
\begin{flalign}
E(B_{r}):=&\mint_{B_{r}}\snr{Dv}^{p_{2}(2r)} \ dx\le \lambda^{-1}\mint_{B_{r}}k_{0}\snr{Du}^{p_{2}(2r)} \ dx \nonumber \\
\le& \frac{ 2^{n-\gamma_{1}}\Lambda}{\lambda\omega_{n}}r^{-p_{2}(2r)}(2r)^{p_{2}(2r)-n}\int_{B_{2r}}(1+\snr{Du}^{2})^{p_{2}(2r)/2} \ dx\le c_{*}(\varepsilon/r)^{p_{2}(2r)}\label{smallv},
\end{flalign}
where we set $c_{*}:=\frac{ 2^{n-\gamma_{1}}\Lambda}{\lambda\omega_{n}}+1$. Now we claim that $v$ is approximately $p_{2}(2r)$-harmonic in the sense of \eqref{136}. This is actually the case: in fact, with reference to the terminology used in Lemma \ref{phihar}, let $\tilde{d}:=d_{2}$, where $d_{2}\in (0,1)$ is the exponent given by Lemma \ref{inpoi}, pick any $\tilde{\theta}\in (0,1)$ and let $\tilde{\delta}=\tilde{\delta}(\tilde{\theta},\tilde{d},p_{2}(2r))$ be the "closeness" parameter appearing in \eqref{136}. Moreover, for reasons that will be clear in a few lines, we also impose a first restriction on the size of $\varepsilon$. Precisely, keeping in mind the definition of $c_{*}$, we ask that
\begin{flalign}\label{har0}
\varepsilon\le \min\left\{\frac{\lambda\omega_{n}}{2^{n-\gamma_{1}}\Lambda+\lambda\omega_{n}},\left(\frac{\tilde{\delta}\lambda}{\Lambda \gamma_{2} c_{\m}+\lambda}\right)^{\frac{\gamma_{1}}{\gamma_{1}-1}}\right\}.
\end{flalign}
By \eqref{elv}, we estimate
\begin{flalign*}
&\left | \ \mint_{B_{r}}p_{2}(2r)\snr{Dv}^{p_{2}(2r)-2}Dv\cdot D\varphi \ dx\  \right |\nonumber \\
&\quad\quad\quad\quad\quad\quad\quad\quad= k_{0}^{-1}\left | \ \mint_{B_{r}}k_{0}p_{2}(2r)\snr{Dv}^{p_{2}(2r)-2}A_{v}(Dv,Dv)\varphi \ dx\  \right |\nonumber\\
&\quad\quad\quad\quad\quad\quad\quad\quad\stackleq{20}\frac{\Lambda \gamma_{2}c_{\m}}{\lambda}\mint_{B_{r}}\snr{Dv}^{p_{2}(2r)}\snr{\varphi} \ dx\le \tilde{c}\nr{D\varphi}_{L^{\infty}(B_{r})}rE(B_{r}),
\end{flalign*}
for all $\varphi \in C^{\infty}_{c}(B_{r},\mathbb{R}^{N})$, where $\tilde{c}:=1+\frac{\Lambda \gamma_{2}c_{\m}}{\lambda}$. For $\delta \in (0,1)$, by Young's inequality: $ab\le \delta a^{p}+\delta^{-\frac{1}{p-1}}b^{p'}$, with exponents $p_{2}(2r)$ and $p_{2}'(2r):=\frac{p_{2}(2r)}{p_{2}(2r)-1}$ we get
\begin{flalign*}
\tilde{c}E(B_{r})r\nr{D\varphi}_{L^{\infty}(B_{r})}\le&\delta^{-\frac{1}{p_{2}(2r)-1}}\tilde{c}^{p'(2r)}(rE(B_{r}))^{p'_{2}(2r)}+\delta \nr{D\varphi}^{p_{2}(2r)}_{L^{\infty}(B_{r})}\nonumber \\
\le &\delta^{-\frac{1}{\gamma_{1}-1}}\tilde{c}^{\gamma_{1}'}(rE(B_{r}))^{p'_{2}(2r)-1}(rE(B_{r}))+\delta \nr{D\varphi}^{p_{2}(2r)}_{L^{\infty}(B_{r})}\nonumber \\
\stackleq{smallv}&\delta^{-\frac{1}{\gamma_{1}-1}}\tilde{c}^{\gamma_{1}'}\left[rc_{*}(\varepsilon/r)^{p_{2}(2r)}\right]^{p_{2}(2r)-1}(rE(B_{r}))+\delta \nr{D\varphi}^{p_{2}(2r)}_{L^{\infty}(B_{r})}\\
\le &\delta^{-\frac{1}{\gamma_{1}-1}}\tilde{c}^{\gamma_{1}'}\left[c_{*}\varepsilon(\varepsilon/r)^{p_{2}(2r)-1}\right]^{p_{2}(2r)-1}(rE(B_{r}))+\delta \nr{D\varphi}^{p_{2}(2r)}_{L^{\infty}(B_{r})}\\
\stackleq{har0}&\delta^{-\frac{1}{\gamma_{1}-1}}\tilde{c}^{\gamma_{1}'}(\varepsilon/r)rE(B_{r})+\delta \nr{D\varphi}^{p_{2}(2r)}_{L^{\infty}(B_{r})}\le \delta^{-\frac{1}{\gamma_{1}-1}}\tilde{c}^{\gamma_{1}'}\varepsilon E(B_{r})+\delta \nr{D\varphi}^{p_{2}(2r)}_{L^{\infty}(B_{r})}.
\end{flalign*}
Choosing $\delta\equiv\tilde{\delta}$ in the previous display, we can conclude that
\begin{flalign*}
\left | \ \mint_{B_{r}}p_{2}(2r)\snr{Dv}^{p_{2}(2r)-2}Dv\cdot D\varphi \ dx\  \right |\stackleq{har0}\tilde{\delta}\mint_{B_{r}}\left(\snr{Dv}^{p_{2}(2r)}+\nr{D\varphi}_{L^{\infty}(B_{r})}^{p_{2}(2r)}\right) \ dx,
\end{flalign*}
thus Lemma \ref{phihar} renders a map $\tilde{h}\in v+W^{1,p_{2}(2r)}_{0}(B_{r},\RN)$ solution to \eqref{har1} with $p\equiv p_{2}(2r)$ and satisfying
\begin{flalign}\label{harrr}
\left(\mint_{B_{r}}\snr{Dv-D\tilde{h}}^{p_{2}(2r)d_{2}} \ dx\right)^{\frac{1}{d_{2}}}\le \tilde{\theta}\mint_{B_{r}}\snr{Dv}^{p_{2}(2r)} \ dx. 
\end{flalign}
Before going on, we would like to stress that $\varepsilon$ depends on $\tilde{\theta}$ as well, since by looking at the dependencies of $d_{2}$, it is evident that $\tilde{\delta}=\tilde{\delta}(n,\gamma_{1},\gamma_{2},\tilde{\theta})$, and \eqref{har0} yields in particular that $\varepsilon=\varepsilon(\tilde{\delta})$. This is not an obstruction, since the value of $\tilde{\theta}$ will be fixed in the next step as a function of $(\texttt{data})$.\\\\
\emph{Step 3: comparison, second time}.
First notice that, since $v$ is a solution to the frozen Dirichlet problem \eqref{fdp}, a Caccioppoli-type inequality holds. In fact, with the same strategy adopted for the proof of Lemma \ref{cacc}, we have
\begin{flalign}\label{caccv}
\mint_{B_{\rr}}\snr{Dv}^{p_{2}(2r)}\le c\mint_{B_{2\rr}}\left |\frac{v-(v)_{2\rr}}{\rr} \right |^{p_{2}(2r)} \ dx,
\end{flalign}
with $c=c(n,N,\m,\lambda,\Lambda,\gamma_{1},\gamma_{2})$ for all balls $B_{2\rr}\Subset B_{r}$. Now fix any $\rr \in (0,r/4)$. According to the previous estimates we can proceed in the following way
\begin{flalign*}
\int_{B_{\rr}}(1+\snr{Du}^{2})^{p_{2}(2r)/2}\le &c\rr^{n}+c\left\{\int_{B_{\rr}}\snr{Du-Dv}^{p_{2}(2r)} \ dx +\rr^{n}\mint_{B_{\rr}}\snr{Dv}^{p_{2}(2r)} \ dx\right\}\\
\stackleq{caccv}&c\rr^{n}+c\left\{\int_{B_{r}}\snr{Du-Dv}^{p_{2}(2r)} \ dx +\rr^{n}\mint_{B_{2\rr}}\left | \frac{v-(v)_{2\rr}}{\rr}\right |^{p_{2}(2r)} \ dx\right\}\\
\stackleq{38}&c\rr^{n}+c\left\{\left[\varepsilon^{\kappa_{3}}+r^{\kappa_{1}}\varepsilon^{\kappa_{2}}\right]\int_{B_{2r}}(1+\snr{Du}^{2})^{p_{2}(2r)/2} \ dx +\rr^{n}\left(\mint_{B_{2\rr}}\snr{Dv}^{p_{2}(2r)d_{2}} \ dx\right)^{\frac{1}{d_{2}}}\right\}\\
\le &c\rr^{n}+c\left[\varepsilon^{\kappa_{3}}+r^{\kappa_{1}}\varepsilon^{\kappa_{2}}\right]\int_{B_{2r}}(1+\snr{Du}^{2})^{p_{2}(2r)/2} \ dx\\
&+c\rr^{n}\left\{\left(\mint_{B_{2\rr}}\snr{Dv-D\tilde{h}}^{p_{2}(2r)d_{2}} \ dx\right)^{\frac{1}{d_{2}}}+\rr^{n}\mint_{B_{2\rr}}\snr{D\tilde{h}}^{p_{2}(2r)} \ dx\right\}\\
\stackleq{harrr}& c\rr^{n}+c\left[\varepsilon^{\kappa_{3}}+r^{\kappa_{1}}\varepsilon^{\kappa_{2}}\right]\int_{B_{2r}}(1+\snr{Du}^{2})^{p_{2}(2r)/2} \ dx\\
&+c(r/\rr)^{n(d_{2}^{-1}-1)}\tilde{\theta}\int_{B_{r}}\snr{Dv}^{p_{2}(2r)} \ dx+c\rr^{n}\mint_{B_{2\rr}}\snr{D\tilde{h}}^{p_{2}(2r)} \ dx\\
\stackleq{140}& c\rr^{n}+c\left[\varepsilon^{\kappa_{3}}+r^{\kappa_{1}}\varepsilon^{\kappa_{2}}+(r/\rr)^{n(d_{2}^{-1}-1)}\tilde{\theta}\right]\int_{B_{2r}}(1+\snr{Du}^{2})^{p_{2}(2r)/2} \ dx\\
&+c(\rr/r)^{n}\int_{B_{r}}\snr{Dv}^{p_{2}(2r)} \ dx\\
\le& c\rr^{n}+c\left[\varepsilon^{\kappa_{3}}+r^{\kappa_{1}}\varepsilon^{\kappa_{2}}+(r/\rr)^{n(d_{2}^{-1}-1)}\tilde{\theta}+(\rr/r)^{n}\right]\int_{B_{2r}}(1+\snr{Du}^{2})^{p_{2}(2r)/2} \ dx,
\end{flalign*}
where we also used Lemma \ref{inpoi}, the minimality of $v$ in the Dirichlet class $W^{1,p_{2}(2r)}_{u}(B_{r},\m)$ and of $\tilde{h}$ in $v+W^{1,p_{2}(2r)}_{0}(B_{r},\RN)$ and the reference estimate in $\eqref{140}_{1}$ which holds for $\tilde{h}$ since $\tilde{h}$ is a solution to \eqref{har1} with $p\equiv p_{2}(2r)$, (set $k_{0}\equiv 1$, $p\equiv p_{2}(2r)$, $g^{\alpha\beta}\equiv \delta^{\alpha\beta}$ and $h_{ij}\equiv \delta_{ij}$ for $\alpha,\beta \in \{1,\cdots,n\}$ and $i,j\in \{1,\cdots,N\}$ in \eqref{rev1}). Here $c=c(\texttt{data})$. For the ease of exposition, let us set $s\equiv2r$. Hence we can rewrite the previous estimate as
\begin{flalign}\label{39}
\int_{B_{\rr}}(1+\snr{Du}^{2})^{p_{2}(s)/2} \ dx \le c_{0}\rr^{n}+c_{1}\left[\varepsilon^{\kappa_{3}}+s^{\kappa_{1}}\varepsilon^{\kappa_{2}}+(s/\rr)^{n(d_{2}^{-1}-1)}\tilde{\theta}+(\rr/s)^{n}\right]\int_{B_{s}}(1+\snr{Du}^{2})^{p_{2}(s)/2} \ dx,
\end{flalign}
for $c_{0}=c_{0}(n,\gamma_{1},\gamma_{2})$ and $c_{1}=c_{1}(\texttt{data})$.\\\\
\emph{Step 4: Morrey-type estimates.} Our goal now is to prove a Morrey type estimate for the energy which will eventually lead to the continuity of solutions. For $\tau \in \left(0,\frac{1}{4}\right)$, we let $\rr\equiv\tau s$ in \eqref{39} and multiply both sides of it by $(\tau s)^{p_{2}(s)-n}$. We then have
\begin{flalign}\label{40}
(\tau s)^{p_{2}(s)-n}\int_{B_{\tau s}}&(1+\snr{Du}^{2})^{p_{2}(s)/2} \ dx\nonumber \\
\le& c_{1}\tau^{p_{2}(s)}\left\{\tau^{-n}(\varepsilon^{\kappa_{3}}+s^{\kappa_{1}}\varepsilon^{\kappa_{2}})+\tau^{-nd_{2}^{-1}}\tilde{\theta}+1\right\}s^{p_{2}(s)-n}\int_{B_{s}}(1+\snr{Du}^{2})^{p_{2}(s)/2} \ dx+c_{0}(\tau s)^{p_{2}(s)}.
\end{flalign}
Adopting the notation introduced in \cite{RT3}, we consider the following quantities:
\begin{flalign*}
&\phi(r,p):=\left(r^{p}\mint_{B_{r}}(1+\snr{Du}^{2})^{p/2} \ dx\right)^{\frac{1}{p}}=\omega_{n}^{-\frac{1}{p}}\left(r^{p-n}\int_{B_{r}}(1+\snr{Du}^{2})^{p/2} \ dx\right)^{\frac{1}{p}},\\
&\psi(r):=\phi(r,p_{2}(r)).
\end{flalign*}
In these terms, \eqref{40} reads as
\begin{flalign}\label{41}
\phi^{p_{2}(s)}(\tau s,p_{2}(s))\le c_{1}\tau^{p_{2}(s)}\left\{\tau^{-n}(\varepsilon^{\kappa_{3}}+s^{\kappa_{1}}\varepsilon^{\kappa_{2}})+\tau^{-nd_{2}^{-1}}\tilde{\theta}+1\right\}\psi^{p_{2}(s)}(s)+c_{0}(\tau s)^{p_{2}(s)},
\end{flalign}
so recalling that
\begin{flalign}\label{122}
\phi(r,p)\le \phi(r,q) \ \ \mbox{for} \ \ p\le q,
\end{flalign}
we obtain from \eqref{41}
\begin{flalign}\label{42}
\psi(\tau s)\le c_{2}\tau\left\{\tau^{-\frac{n}{p_{2}(s)}}\left[\varepsilon^{\frac{\kappa_{3}}{p_{2}(s)}}+s^{\frac{\kappa_{1}}{p_{2}(s)}}\varepsilon^{\frac{\kappa_{2}}{p_{2}(s)}}\right]+\tau^{-\frac{n}{d_{2}p_{2}(s)}}\tilde{\theta}^{\frac{1}{p_{2}(s)}}+1\right\}\psi(s)+c_{3}(\tau s),
\end{flalign}
where $c_{2}=c_{2}(\texttt{data})$ and $c_{3}=c_{3}(n,\gamma_{1},\gamma_{2})$. Since $\tau$, $s$, $\varepsilon$, $\tilde{\theta}\in (0,1)$, from the definitions of $\kappa_{1}$, $\kappa_{2}$ and $\kappa_{3}$ we have
\begin{flalign*}
\tau^{-\frac{n}{p_{2}(s)}}\le \tau^{-\frac{n}{\gamma_{1}}}, \quad \tau^{-\frac{n}{d_{2}p_{2}(s)}}\le \tau^{-\frac{n}{d_{2}\gamma_{1}}}, \quad s^{\frac{\kappa_{1}}{p_{2}(s)}}\le s^{\tilde{\kappa}_{1}}, \quad \varepsilon^{\frac{\kappa_{2}}{p_{2}(s)}}\le \varepsilon^{\tilde{\kappa}_{2}}, \quad \varepsilon^{\frac{\kappa_{3}}{p_{2}(s)}}\le \varepsilon^{\tilde{\kappa}_{3}},\quad \tilde{\theta}^{\frac{1}{p_{2}(s)}}\le \tilde{\theta}^{\frac{1}{\gamma_{2}}},
\end{flalign*}
where
\begin{flalign*}
\tilde{\kappa}_{1}:=\begin{cases} 
\ \frac{\kappa}{\gamma_{2}}\quad & 2\le p_{2}(s)\\
\ \frac{\kappa}{2}\quad & 1<p_{2}(s)<2
\end{cases}, \quad 
\tilde{\kappa}_{2}:=\begin{cases} 
\ \frac{\gamma_{1}\sigma}{\gamma_{2}}\quad & 2\le p_{2}(s)\\
\ \frac{\gamma_{1}\sigma}{2}\quad & 1<p_{2}(s)<2
\end{cases}, \quad \tilde{\kappa}_{3}:=\begin{cases} 
\ \frac{\gamma_{1}\sigma}{\gamma_{2}(1+\sigma)}\quad & 2\le p_{2}(s)\\
\ \frac{\gamma_{1}\sigma}{2(1+\sigma)}\quad & 1<p_{2}(s)<2
\end{cases},
\end{flalign*}
thus \eqref{42} becomes
\begin{flalign}\label{43}
\psi(\tau s)\le c_{2}\tau\left\{\tau^{-\frac{n}{\gamma_{1}}}\left(\varepsilon^{\tilde{\kappa}_{3}}+s^{\tilde{\kappa}_{1}}\varepsilon^{\tilde{\kappa}_{2}}\right)+\tau^{-\frac{n}{d_{2}\gamma_{1}}}\tilde{\theta}^{\frac{1}{\gamma_{2}}}+1\right\}\psi(s)+c_{3}(\tau s).
\end{flalign}
Now we need to make a proper choice of the parameters appearing in \eqref{43}. We first select any $\beta \in (0,1)$ and an $\eta\in (\beta,1)$ and ask that $c_{2}\tau \le\tau^{\eta}/5$, thus obtaining
\begin{flalign*}
\psi(\tau s)\le (\tau^{\eta}/5)\left\{\tau^{-\frac{n}{\gamma_{1}}}\left(\varepsilon^{\tilde{\kappa}_{3}}+s^{\tilde{\kappa}_{1}}\varepsilon^{\tilde{\kappa}_{2}}\right)+\tau^{-\frac{n}{d_{2}\gamma_{1}}}\tilde{\theta}^{\frac{1}{\gamma_{2}}}+1\right\}\psi(s)+c_{3}(\tau s)^{\beta}.
\end{flalign*}
Moreover, we require that the threshold radius $R_{*}$ introduced at the beginning of \emph{Step 1} satisfies $c_{3}(\tau R_{*})^{\beta}\le (\varepsilon/5)$. Finally we recall that $\tilde{\theta}$ is arbitrary, therefore we fix $\tilde{\theta}=2^{-\gamma_{2}}\tau^{\frac{n\gamma_{2}}{d_{2}\gamma_{1}}}$ and, since $\tilde{\kappa}_{2}\ge \tilde{\kappa}_{3}$ renders $\varepsilon^{\tilde{\kappa}_{2}}\le \varepsilon^{\tilde{\kappa}_{3}}$, the choice 
\begin{flalign*}
\varepsilon\le \left\{\left(\frac{ 2^{n-\gamma_{1}}\Lambda}{\lambda\omega_{n}}+1\right)^{-1},\left(\frac{\tilde{\delta}\lambda}{\Lambda \gamma_{2} c_{\m}+\lambda}\right)^{\frac{\gamma_{1}}{\gamma_{1}-1}},\tau^{\frac{n}{\gamma_{1}\tilde{\kappa}}}\right\}
\end{flalign*}
and \eqref{small} allow concluding that
\begin{flalign}\label{44}
\psi(\tau s)\le \tau^{\eta}\psi(s)+c_{3}(\tau s)^{\beta}\quad \mbox{and}\quad \psi(\tau s) \le\frac{4}{5}\tau^{\eta}\psi(s)+\frac{\varepsilon}{5}<\varepsilon.
\end{flalign}
We remark that, since $\eta$ is ultimately influenced only by the choice of $\beta$, we can incorporate the dependency from $\eta$ in the one from $\beta$, so the above procedure defines the following dependencies: $\tau=\tau(\texttt{data},\beta)$, $\varepsilon=(\texttt{data},\beta)$ and $R_{*}=R_{*}(\texttt{data},\beta)$. Estimate $\eqref{44}_{2}$ legalizes iterations, so we can repeat $\eqref{44}_{1}$ replacing $s$ by $\tau s$, $\tau^{2}s$, $\tau^{3}s$, $\cdots$ to get
\begin{flalign}\label{45}
\psi(\tau^{j+1}s)\le \tau^{\eta(j+1)}\psi(s)+c_{3}s^{\beta}\tau^{\beta (j+1)}\sum_{i=0}^{j}\tau^{i(\eta-\beta)}\le \tau^{(j+1)\eta}\psi(s)+c_{4}s^{\beta}\tau^{(j+1)\beta},
\end{flalign}
for $c_{4}=c_{4}(\texttt{data},\beta)$. Now, for any $\varsigma\in (0,s/8)$ we can find an integer $j\ge 1$ with $\tau^{j+1}s<\varsigma\le \tau^{j}s$, so, from \eqref{45} we obtain
\begin{flalign}\label{46}
\psi(\varsigma)\le &\tau^{1-\frac{n}{\gamma_{1}}}\psi(\tau^{j}s)\le \tau^{1-\frac{n}{\gamma_{1}}}\left\{\tau^{\eta j}\psi(s)+c_{4}s^{\beta}\tau^{\beta j}\right\}\nonumber \\
\le &\tau^{-\frac{n}{\gamma_{1}}}\left\{(\varsigma/s)^{\eta}\psi(s)+c_{4}s^{\beta}(\varsigma/s)^{\beta}\right\}\nonumber \\
\stackleq{small}&\tau^{-\frac{n}{\gamma_{1}}}\left\{(\varsigma/s)^{\beta}\varepsilon+c_{4}\varsigma^{\beta}\right\}\le c_{5}(\varsigma/s)^{\beta},
\end{flalign}
with $c_{5}=c_{5}(\texttt{data},\beta)$, while if $\varsigma\in [s/8,s)$, since $(s/\varsigma)\le 8$, there obviously holds
\begin{flalign}\label{n2}
\psi(\varsigma)\le 8^{\frac{n}{\gamma_{1}}+\beta-1}(\varsigma/s)^{\beta}\left\{\psi(s)+s^{\beta}\right\}\le 8^{\frac{n}{\gamma_{1}}+\beta-1}(\varsigma/s)^{\beta}(\varepsilon+s^{\beta})\le c(\varsigma/s)^{\beta},
\end{flalign}
again for $c=c(\texttt{data},\beta)$. We can actually improve estimates \eqref{46}-\eqref{n2} by getting rid of the restriction $s\le R_{*}$; we shall only retain $s\le \tilde{R}_{*}/2$. In fact, in case $0<\varsigma<R_{*}\le s\le \tilde{R}_{*}/2$, we see that
\begin{flalign}\label{n0}
\psi(\varsigma)\stackleq{46}c_{5}(\varsigma/R_{*})^{\beta}\le c(\varsigma/s)^{\beta},
\end{flalign}
with $c=c(\texttt{data},\beta)$, while, if $0<R_{*}<\varsigma\le s\le \tilde{R}_{*}/2$ we have
\begin{flalign}\label{n1}
\psi(\varsigma)\le &(s/\varsigma)^{\frac{n}{\gamma_{1}}-1}\psi(s)\le R_{*}^{1-\frac{n}{\gamma_{1}}}(R_{*}/\varsigma)^{\frac{n}{\gamma_{1}}-1+\beta}(\varsigma/s)^{\beta}\psi(s)\le c(\varsigma/s)^{\beta},
\end{flalign}
again for $c=c(\texttt{data},\beta)$. Now, by the continuity of Lebesgue's integral and of the mapping $y\mapsto p_{2}(y,s):= \sup_{x\in B_{s}(y)}p(x)$, we can conclude that if \eqref{small} holds for $B_{s}(x_{0})$ then it holds also on $B_{s}(y)$ for all $y$ belonging to a sufficiently small neighborhood of $x_{0}$. Hence, if we let
\begin{flalign*}
D_{0}:=\left\{y \in B_{\tilde{R}_{*}}(x_{0})\colon \left( s^{p_{2}(y,s)-n}\int_{B_{s}(y)}(1+\snr{Du}^{2})^{p_{2}(y,s)/2} \ dx\right)^{\frac{1}{p_{2}(y,s)}}<\varepsilon \ \mbox{for some} \ s\le \tilde{R}_{*}/2; \  B_{s}(y)\Subset B_{\tilde{R}_{*}}(x_{0})\right\},
\end{flalign*}
we see that it is open and, taking radii $0<\varsigma<s$, for $y \in D_{0}$ we have
\begin{flalign}\label{47}
\varsigma^{-n+\gamma_{1}-\beta\gamma_{1}}\int_{B_{\varsigma}(y)}\snr{Du}^{\gamma_{1}} \ dx=&\omega_{n}\varsigma^{-\beta\gamma_{1}}\phi^{\gamma_{1}}(\varsigma,\gamma_{1})\stackleq{122}\omega_{n}\varsigma^{-\beta\gamma_{1}}\psi^{\gamma_{1}}(\varsigma)\stackleq{46}c_{6}s^{-\beta\gamma_{1}},
\end{flalign}
for $c_{6}=c_{6}(\texttt{data},\beta)$, so, from Morrey's growth theorem, $u \in C^{0,\beta}(D_{0},\m)$. We would like to stress that $\beta$ is an arbitrary number in $(0,1)$ and, being the interval open it is always possible to find $\eta\in (0,1)$ so that $\beta<\eta<1$. Hence we can take any $\beta\in (0,1)$ in the above estimates and deduce from \eqref{47} that actually $u \in C^{0,\beta}(D_{0},\m)$ for all $\beta \in (0,1)$. Of course, the values of all the parameters involved will change accordingly to the one of $\beta$ (and consequently of $\eta$) we choose. After a standard covering argument, we obtain that $u \in C^{0,\beta}_{\mathrm{loc}}(\Omega_{0},\m)$ for any $\beta \in (0,1)$. Now consider any open subset $\tilde{\Omega}\Subset \Omega_{0}$. From \eqref{46} and a standard covering argument, we also obtain the Morrey type estimate
\begin{flalign}\label{mor}
\mint_{B_{\varsigma}}\snr{Du}^{p_{2}(\varsigma)} \ dx \le c_{7}\varsigma^{-\gamma_{2}(1-\beta)},
\end{flalign}
for all $B_{\varsigma}\Subset \tilde{\Omega}$, $\varsigma\le \tilde{R}_{*}/2$ and any $\beta \in (0,1)$. Here $c_{7}=c_{7}(\texttt{data},\beta)$.\\\\
\emph{Step 5: Hausdorff dimension of the Singular Set.} Given the characterization of $D_{0}$, we easily see that, if $\Sigma_{0}(u,B_{\tilde{R}_{*}}(x_{0})):=B_{\tilde{R}_{*}}(x_{0})\setminus D_{0}$, then
\begin{flalign*}
\Sigma_{0}(u,B_{\tilde{R}_{*}}(x_{0}))\subset \left\{y \in B_{\tilde{R}_{*}}(x_{0})\colon \ \limsup_{s\to 0}\left(s^{p_{2}(y,s)-n}\int_{B_{s}(y)}\snr{Du}^{p_{2}(y,s)} \ dy\right)^{\frac{1}{p_{2}(y,s)}}>0\right\}.
\end{flalign*}
Now, if $p_{m}(x_{0},\tilde{R}_{*}):=\inf_{x\in B_{\tilde{R}_{*}}(x_{0})}p(x)$, then, as in \eqref{kom},
\begin{flalign}\label{61}
p_{2}(y,s)\le (1+\sigma_{0})p_{m}(x_{0},\tilde{R}_{*}) \ \ \mbox{for all} \ \ 0<s\le \tilde{R}_{*}/2, \ \ B_{s}(y)\Subset B_{\tilde{R}_{*}}(x_{0}),
\end{flalign}
so we obtain,
\begin{flalign*}
\left(s^{p_{2}(y,s)}\mint_{B_{s}(y)}\snr{Du}^{p_{2}(y,s)} \ dx\right)^{\frac{1}{p_{2}(y,s)}} \stackleqq{61}&\left(s^{p_{m}(x_{0},\tilde{R}_{*})(1+\sigma_{0})}\mint_{B_{s}(y)}\snr{Du}^{p_{m}(x_{0},\tilde{R}_{*})(1+\sigma_{0})} \ dx\right)^{\frac{1}{p_{m}(x_{0},\tilde{R}_{*})(1+\sigma_{0})}}\\
\leq&  \left(s^{p_{m}(x_{0},\tilde{R}_{*})(1+\delta_{0})}\mint_{B_{s}(y)}\snr{Du}^{p_{m}(x_{0},\tilde{R}_{*})(1+\delta_{0})} \ dx\right)^{\frac{1}{p_{m}(x_{0},\tilde{R}_{*})(1+\delta_{0})}},
\end{flalign*}
where we also used that $\sigma_{0}<\delta_{0}$. Hence, if $y \in B_{\tilde{R}_{*}}(x_{0})$ is such that
\begin{flalign*}
0<\limsup_{s\to 0}\left(s^{p_{2}(y,s)-n}\int_{B_{s}(y)}\snr{Du}^{p_{2}(y,s)} \ dx\right)^{\frac{1}{p_{2}(y,s)}},
\end{flalign*}
then 
\begin{flalign*}
0<\limsup_{s\to 0}\left(s^{p_{m}(x_{0},\tilde{R}_{*})(1+\delta_{0})-n}\int_{B_{s}(y)}\snr{Du}^{p_{m}(x_{0},\tilde{R}_{*})(1+\delta_{0})} \ dx\right)^{\frac{1}{p_{m}(x_{0},\tilde{R}_{*})(1+\delta_{0})}}.
\end{flalign*}
This allows to conclude that
\begin{flalign*}
\Sigma_{0}(u,B_{\tilde{R}_{*}}(x_{0}))\subset \left\{y \in B_{\tilde{R}_{*}(x_{0})}\colon \limsup_{s\to 0}\left(s^{p_{m}(x_{0},\tilde{R}_{*})(1+\delta_{0})-n}\int_{B_{s}(y)}\snr{Du}^{p_{m}(x_{0},\tilde{R}_{*})(1+\delta_{0})} \ dx\right)^{\frac{1}{p_{m}(x_{0},\tilde{R}_{*})(1+\delta_{0})}}>0\right\}=:D_{1}.
\end{flalign*}
By \cite[Proposition 2.7]{giu} it follows that $\dim_{\mathcal{H}}(D_{1})\le n-p_{m}(x_{0},\tilde{R}_{*})(1+\delta_{0})$. Now, covering $\Omega$ with balls having the same features of $B_{\tilde{R}_{*}}(x_{0})$ and remembering that $p_{m}(x_{0},\tilde{R}_{*})\ge \gamma_{1}$, we obtain that
$
\dim_{\mathcal{H}}(\Sigma_{0}(u))\le n-\gamma_{1}(1+\delta_{0})<n-\gamma_{1},
$
and so $\dim_{\mathcal{H}}(\Sigma_{0}(u))<n-\gamma_{1}$.
\\\\
\emph{Step 6: partial $C^{1,\beta_{0}}$-regularity.} 
In this part we follow the approach of \cite[Theorem 3.1]{harlin}. So far we know that the regular set $\Omega_{0}\subset \Omega$ is a relatively open set of full $n$-dimensional Lebesgue measure and $u\in C^{0,\beta}_{\mathrm{loc}}(\Omega_{0},\m)$ for all $\beta \in (0,1)$. For reasons that will be clear in a few lines, we fix
\begin{flalign}\label{beta}
\tilde{\beta}:=\max\left\{\frac{1}{2},1-\frac{1}{4\gamma_{2}}\min\left\{\frac{1}{2},\alpha-n\sigma_{0},\frac{\gamma_{1}\sigma}{2(1+\sigma)}\right\}\right\}\in (0,1),
\end{flalign}
where $\sigma_{0}$, $\sigma$ are as in \eqref{sigma0}-\eqref{sigma} respectively, and two open subsets $\tilde{\Omega}\Subset\Omega'\Subset \Omega_{0}$. Given the expression of $\tilde{\beta}$, we shall incorporate any dependency from $\tilde{\beta}$ of the constants appearing in the forthcoming estimates into the one from $(\texttt{data})$. We cover $\tilde{\Omega}$ with finitely many balls contained in $\Omega'$, (with size and number depending only on $\m$, $[u]_{0,\tilde{\beta};\Omega'}$ and on $\diam(\Omega')$), whose image lies in small coordinate neighborhoods of $\m$. Precisely, by the continuity of $u$ and up to scaling, rotating and translating $\m$ we can now assume that $u(\tilde{\Omega})$ is contained into the image of a single chart $f(B_{1}^{m})$, so we can find an $\omega\colon \tilde{\Omega}\to \mathbb{R}^{m}$ such that $u=f(\omega)$ and $\snr{\omega}\le 1$. Here $f\colon \mathbb{R}^{m}\mapsto \m$ is such that
\begin{flalign}\label{f}
\nr{\nabla f}_{L^{\infty}(B^{m}_{4m})}\le c(\m),\quad \nr{\nabla^{2} f}_{L^{\infty}(B_{4m}^{m})}\le c(\m) \quad \mathrm{and} \quad \snr{\nabla (f^{-1})(u)}\le c(\m).
\end{flalign}
The above conditions are for instance satisfied by the
inverse of the stereographic projection
\begin{flalign*}
S\colon \mathbb{R}^{N-1}\ni y\mapsto \left(\frac{\snr{y}^{2}-1}{\snr{y}^{2}+1},\frac{2y}{\snr{y}^{2}+1}\right)\in \SN,
\end{flalign*}
see \cite{demi,harlin}. From $\eqref{f}_{3}$ we get that
\begin{flalign}\label{f13}
\int_{U}\snr{D\omega}^{p(x)} \ dx \le c\int_{U}\snr{Du}^{p(x)} \ dx,
\end{flalign}
for any $U\subseteq \tilde{\Omega}$, with $c=c(\m,\gamma_{1},\gamma_{2})$. Since $u$ is an $\m$-constrained local minimizer of \eqref{cvp}, then $\omega$ minimizes the variational integral
\begin{flalign}\label{uvp}
W^{1,p(\cdot)}(\tilde{\Omega},\mathbb{R}^{m})\ni \zeta \mapsto \mathcal{H}(\zeta,\tilde{\Omega}):=\int_{\tilde{\Omega}}k(x)(\delta^{\alpha\beta}h_{ij}(\zeta)D_{\alpha}\zeta^{i}D_{\beta}\zeta^{j})^{p(x)/2} \ dx,
\end{flalign}
where $(\delta^{\alpha\beta})_{\alpha\beta}$ is the $n\times n$ identity matrix and $(h_{ij})_{ij}$ is the $m\times m$ symmetric matrix $((\nabla f)^{T}\nabla f)_{ij}$. From \eqref{f} and being $f$ a chart, $(h_{ij})_{ij}$ is uniformly elliptic and uniformly bounded, in the sense that 
\begin{flalign*}
\sup_{i,j \in \{1,\cdots,m\}}\nr{h_{ij}}_{L^{\infty}(B^{m}_{4m})}<c \quad \mbox{and} \quad c_{1}\snr{\zeta}^{2}\le h_{ij}(y)\zeta^{i}\zeta^{j}\le c_{2}\snr{\zeta}^{2}
\end{flalign*}
for all $\zeta \in \mathbb{R}^{m\times m}$, whenever $\snr{y}\le 4m$. Here $c$, $c_{1}$, $c_{2}$ depend only on $\m$. Given the previous considerations, it is easy to see that the integrand
\begin{flalign*}
H(x,y,z):=k(x)(\delta^{\alpha\beta}h_{ij}(y)z_{\alpha}^{i}z_{\beta}^{j})^{p(x)/2}
\end{flalign*}
satisfies the following set of conditions:
\begin{flalign}\label{assh}
\begin{cases}
\ c_{1}\snr{z}^{p(x)}\le H(x,y,z)\le c_{2}\snr{z}^{p(x)} \\
\ \snr{H(x_{1},y,z)-H(x_{2},y,z)}\le c_{\varepsilon_{0}}\snr{x_{1}-x_{2}}^{\alpha}\left(1+\snr{z}^{(1+\varepsilon_{0})\max\{p(x_{1}),p(x_{2})\}}\right) \ \ \mbox{for any }\varepsilon_{0}>0\\
\ \snr{H(x,y_{1},z)-H(x,y_{2},z)}\le c\snr{y_{1}-y_{2}}\snr{z}^{p(x)}\\
\ \snr{\partial H(x,y,z)}\snr{z}+\snr{\partial^{2}H(x,y,z)}\snr{z}^{2}\le c\snr{z}^{p(x)}\\
\ \langle\partial^{2}H(x,y,z)\xi,\xi\rangle\ge c\snr{z}^{p(x)-2}\snr{\xi}^{2},
\end{cases}
\end{flalign}
for $\snr{y}\le 4m$. Here, all the constants depend only on $m$, $\m$, $\lambda$, $\Lambda$, $\gamma_{1}$, $\gamma_{2}$, $[k]_{0,\alpha}$, $[p]_{0,\alpha}$ and $\alpha$, except for $c_{\varepsilon_{0}}$, which, in addition, depends also from $\varepsilon_{0}$. In particular, from $\eqref{assh}_{1}$, we see that $\omega$ minimizes a functional controlled from below and above by the $p(\cdot)$-laplacean energy, so there is no loss of generality in assuming that Lemmas \ref{cacc} and \ref{inngeh} (and \ref{bougeh} for the associated frozen problem) hold true with the same parameters as before. Moreover, $\eqref{f}_{3}$, \eqref{f13} and \eqref{mor} allow transferring regularity from $u$ to $\omega$. In fact we have
\begin{flalign}\label{morom}
\omega \in C^{0,\beta}(\tilde{\Omega},\mathbb{R}^{m}), \quad [\omega]_{0,\beta;\tilde{\Omega}}\le c(\m)[u]_{0,\beta;\tilde{\Omega}}, \quad \mint_{B_{\rr}}\snr{D\omega}^{p_{2}(\rr)} \ dx\le c\rr^{-\gamma_{2}(1-\beta)},
\end{flalign}
for any $\beta \in (0,1)$ and all $B_{\rr}\Subset \tilde{\Omega}$. Notice that, by \eqref{f13} and $\eqref{morom}_{2}$ we can incorporate any dependency from $\nr{(\snr{D\omega})^{p(\cdot)}}_{L^{1}(\tilde{\Omega})}$ or from $[\omega]_{0,\beta;\tilde{\Omega}}$ of the constants in the forthcoming estimates into the one from $\nr{(\snr{Du})^{p(\cdot)}}_{L^{1}(\tilde{\Omega})}$ or from $[u]_{0,\beta;\tilde{\Omega}}$. In \eqref{morom} we are going to choose $\beta\equiv\tilde{\beta}$, where $\tilde{\beta}$ is as in \eqref{beta}. Let $\sigma_{0}$, $\tilde{R}_{*}$ and $\sigma$ be as in \eqref{sigma0}, \eqref{18} and \eqref{sigma} respectively and fix any ball $B_{\rr}\Subset B_{\tilde{R}_{*}} \Subset \tilde{\Omega}$, $\rr\le \tilde{R}_{*}/2$ and let $\vartheta \in W^{1,p_{2}(\rr)}(B_{\rr/4},\mathbb{R}^{m})$ be the solution to the frozen Dirichlet problem
\begin{flalign}\label{fudp}
\omega+W^{1,p_{2}(\rr)}_{0}(B_{\rr/4},\mathbb{R}^{m})\ni \zeta \mapsto \min \int_{B_{\rr/4}}k_{0}(\delta^{\alpha\beta}h_{ij}((\omega)_{\rr/4})D_{\alpha}\zeta^{i}D_{\beta}\zeta^{j})^{p(\rr)/2} \ dx,
\end{flalign}
where $k_{0}$ is the value of $k(\cdot)$ in the centre of $B_{\rr/4}$. For simplicity, define $H_{0}(y,z):=k_{0}(\delta^{\alpha\beta}h_{ij}(y)z_{\alpha}^{i}z_{\beta}^{j})^{p_{2}(\rr)/2}$, and notice that, since $\snr{(\omega)_{\rr/4}}\le 1$, then the integrand $H_{0}((\omega)_{\rr/4},z)$ is of the type covered by Proposition \ref{rif}, see \cite{acefus1,giagiudiff,giamod}. Furthermore, given the specific structure of the integrand, the Maximum principle in \cite{dotleomus} applies, thus $\sup_{x\in B_{\rr/4}}\snr{\vartheta(x)}\le m$. By \eqref{kom}, $\omega$ is an admissible competitor for $\vartheta$ in problem \eqref{fudp} and, as a consequence
\begin{flalign}\label{uelv}
\int_{B_{\rr/4}}\partial H_{0}((\omega)_{\rr/4},D\vartheta)(D\omega-D\vartheta) \ dx =0.
\end{flalign}
Taking into account $\eqref{assh}_{5}$ (with $p_{2}(\rr)$ instead of $p(x)$) and \eqref{uelv} we then estimate
\begin{flalign*}
c\int_{B_{\rr/4}}&(\snr{D\omega}^{2}+\snr{D\vartheta}^{2})^{\frac{p_{2}(\rr)-2}{2}}\snr{D\omega-D\vartheta}^{2} \ dx +c\int_{B_{\rr/4}}\partial H_{0}((\omega)_{\rr/4},D\vartheta)(D\omega-D\vartheta) \ dx\\
=&c\int_{B_{\rr/4}}(\snr{D\omega}^{2}+\snr{D\vartheta}^{2})^{\frac{p_{2}(\rr)-2}{2}}\snr{D\omega-D\vartheta}^{2} \ dx \le \int_{B_{\rr/4}}H_{0}((\omega)_{\rr/4},D\omega)-H_{0}((\omega)_{\rr/4},D\vartheta) \ dx\\
=&\int_{B_{\rr/4}}H_{0}((\omega)_{\rr/4},D\omega)-H(x,(\omega)_{\rr/4},D\omega) \ dx+\int_{B_{\rr/4}}H(x,(\omega)_{\rr/4},D\omega)-H(x,\omega,D\omega) \ dx \\
&+\int_{B_{\rr/4}}H(x,\omega,D\omega)-H(x,\vartheta,D\vartheta) \ dx +\int_{B_{\rr/4}}H(x,\vartheta,D\vartheta)-H(x,(\vartheta)_{\rr/4},D\vartheta) \ dx\\
&+\int_{B_{\rr/4}}H(x,(\vartheta)_{\rho/4},D\vartheta)-H_{0}((\vartheta)_{\rr/4},D\vartheta) \ dx +\int_{B_{\rr/4}}H_{0}((\vartheta)_{\rr/4},D\vartheta)-H_{0}((\omega)_{\rr/4},D\vartheta) \ dx=\sum_{i=1}^{6}\mbox{(I)}_{i},
\end{flalign*}
where $c=c(m,\m,\lambda,\Lambda,\gamma_{1},\gamma_{2},[k]_{0,\alpha},[p]_{0,\alpha})$. Before estimating terms $\mbox{(I)}_{1}$-$\mbox{(I)}_{6}$, let us take care of some quantities which will be recurrent in the forthcoming estimates. By $\eqref{morom}_{1,2}$, we easily have
\begin{flalign}
&\sup_{x\in B_{\rr/4}}\snr{\omega(x)-(\omega)_{\rr/4}}\le c\rr^{\tilde{\beta}}\label{gr0},
\end{flalign}
with $c=c(\m,[u]_{0,\tilde{\beta};\tilde{\Omega}})$. Moreover, it follows from the convex-hull property in \cite{dotleomus} that
\begin{flalign*}
\sup_{x,y\in B_{\rr/4}}\snr{\vartheta(x)-\vartheta(y)}\le \sup_{x,y\in \partial B_{\rr/4}}\snr{\omega(x)-\omega(y)}\stackrel{\eqref{morom}_{2}}{\le} c\rr^{\tilde{\beta}},
\end{flalign*}
for $c=c(\m,[u]_{0,\tilde{\beta};\tilde{\Omega}})$, therefore
\begin{flalign}\label{gr1}
\sup_{x \in B_{\rr/4}}\snr{\vartheta(x)-(\vartheta)_{\rr/4}}\le c\rr^{\tilde{\beta}}.
\end{flalign}
Finally, from Poincar\'e's and H\"older's inequalities, Lemmas \ref{inngeh}, \ref{L0} (\textit{ii.}), \ref{cacc} and by the minimality of $\vartheta$ we see that
\begin{flalign}
\mint_{B_{\rr/4}}\snr{(\omega)_{\rr/4}-(\vartheta)_{\rr/4}}^{p_{2}(\rr)} \ dx\stackrel{}{\le}&c\rr^{p_{2}(\rr)}\mint_{B_{\rr}/4}\snr{D\omega}^{p_{2}(\rr)} \ dx\stackleqq{kom} c\rr^{p_{2}(\rr)}\left(\mint_{B_{\rr/4}}\snr{D\omega}^{(1+\sigma_{0})p(x)} \ dx\right)^{\frac{p_{2}(\rr)}{(1+\sigma_{0})p_{1}(\rr)}}\nonumber \\
\stackrel{}{\le}&c\rr^{p_{2}(\rr)}\left(\mint_{B_{\rr/2}}(1+\snr{D\omega}^{2})^{p(x)/2} \ dx\right)^{\frac{p_{2}(\rr)-p_{1}(\rr)}{p_{1}(\rr)}}\mint_{B_{\rr/2}}(1+\snr{D\omega}^{2})^{p(x)/2} \ dx\nonumber \\
\stackrel{}{\le}&c\rr^{p_{2}(\rr)}\mint_{B_{\rr}}1+\left |\frac{\omega-(\omega)_{\rr}}{\rr} \right |^{p(x)} \ dx\stackrel{\eqref{morom}_{1,2}}{\le}c\rr^{p_{2}(\rr)+(\tilde{\beta}-1)p_{2}(\rr)}= c\rr^{\tilde{\beta}p_{2}(\rr)},\label{gr2}
\end{flalign}
with $c=c(\texttt{data},\nr{(\snr{Du})^{p(\cdot)}}_{L^{1}(\tilde{\Omega})},[u]_{0,\tilde{\beta};\tilde{\Omega}})$. From $\eqref{assh}_{2}$ with $\varepsilon_{0}\equiv\sigma_{0}/2$ and Lemma \ref{inngeh} we get
\begin{flalign}\label{gr3}
\snr{\mbox{(I)}_{1}}\le& c\rr^{\alpha+n}\mint_{B_{\rr/4}}1+\snr{D\omega}^{(1+\varepsilon_{0})p_{2}(\rr)} \ dx\stackleq{kom} c\rr^{\alpha+n}\mint_{B_{\rr/4}}(1+\snr{D\omega}^{2})^{(1+\sigma_{0})p(x)/2} \ dx\nonumber \\
\le &c\rr^{\alpha+n}\left(\mint_{B_{\rr/2}}(1+\snr{D\omega})^{p(x)/2} \ dx\right)^{\sigma_{0}}\mint_{B_{\rr}}(1+\snr{D\omega}^{2})^{p_{2}(\rr)/2} \ dx\nonumber \\
\le& c\rr^{\bar{\kappa}_{1}}\int_{B_{\rr}}(1+\snr{D\omega}^{2})^{p_{2}(\rr)/2} \ dx,
\end{flalign}
where $c=c(\texttt{data},\nr{(\snr{Du})^{p(\cdot)}}_{L^{1}(\tilde{\Omega})})$ and $\bar{\kappa}_{1}:=\alpha-n\sigma_{0}>0$ by \eqref{sigma0}. Now, from $\eqref{assh}_{3}$ and \eqref{gr0} we have
\begin{flalign}\label{gr4}
\snr{\mbox{(I)}_{2}}\le &c\int_{B_{\rr/4}}\snr{\omega-(\omega)_{\rr/4}}\snr{D\omega}^{p(x)} \ dx \le c\rr^{\tilde{\beta}}\int_{B_{\rr}}(1+\snr{D\omega}^{2})^{p_{2}(\rr)/2} \ dx,
\end{flalign}
for $c=c(\texttt{data},[u]_{0,\tilde{\beta};\tilde{\Omega}})$. Since $\omega$ is a local minimizer of \eqref{uvp}, then
\begin{flalign}\label{gr5}
\mbox{(I)}_{3}\le 0.
\end{flalign}
Concerning term $\mbox{(I)}_{4}$, we use $\eqref{assh}_{3}$, \eqref{gr1} and the minimality of $\vartheta$ to bound
\begin{flalign}\label{gr6}
\snr{\mbox{(I)}_{4}}\le& c\int_{B_{\rr/4}}\snr{\vartheta-(\vartheta)_{\rr/4}}\snr{D\vartheta}^{p(x)} \ dx\nonumber\\
\le& c\rr^{\tilde{\beta}}\int_{B_{\rr/4}}(1+\snr{D\vartheta}^{2})^{p_{2}(\rr)/2} \ dx\le c\rr^{\tilde{\beta}}\int_{B_{\rr}}(1+\snr{D\omega}^{2})^{p_{2}(\rr)/2} \ dx,
\end{flalign}
with $c=c(\texttt{data},[u]_{0,\tilde{\beta};\tilde{\Omega}})$. To take care of term $\mbox{(I)}_{5}$, we use $\eqref{assh}_{2}$ with $\varepsilon_{0}\equiv\sigma$ again together with the minimality of $\vartheta$ and Lemmas \ref{bougeh}-\ref{inngeh} to obtain
\begin{flalign}\label{gr7}
\snr{\mbox{(I)}_{5}}\le& c\rr^{\alpha+n}\mint_{B_{\rr/4}}1+\snr{D\vartheta}^{(1+\sigma)p_{2}(\rr)} \ dx\le c\rr^{\alpha+n}\mint_{B_{\rr/4}}1+\snr{D\omega}^{(1+\sigma)p_{2}(\rr)} \ dx\nonumber\\
\le &c\rr^{\alpha+n}\mint_{B_{\rr/4}}1+\snr{D\omega}^{(1+\sigma_{0}/2)p_{2}(\rr)} \ dx \le cr^{\alpha+n}\left(\mint_{B_{\rr/2}}(1+\snr{D\omega}^{2})^{p(x)/2} \ dx\right)^{1+\sigma_{0}}\nonumber \\
\le &c\rr^{\bar{\kappa}_{1}}\int_{B_{\rr}}(1+\snr{D\omega}^{2})^{p_{2}(\rr)/2} \ dx,
\end{flalign}
with $c=c(\texttt{data},\nr{(\snr{Du})^{p(\cdot)}}_{L^{1}(\tilde{\Omega})})$. Finally, by $\eqref{assh}_{3}$, Lemmas \ref{bougeh}-\ref{inngeh} and \eqref{gr2} we obtain
\begin{flalign}\label{gr8}
\snr{\mbox{(I)}_{6}}\le &c\rr^{n}\mint_{B_{\rr/4}}\snr{(\omega)_{\rr/4}-(\vartheta)_{\rr/4}}\snr{D\vartheta}^{p_{2}(\rr)} \ dx\nonumber\\
\le& c\rr^{n}\left(\mint_{B_{\rr/4}}\snr{(\omega)_{\rr/4}-(\vartheta)_{\rr/4}}^{p_{2}(\rr)} \ dx\right)^{\frac{\sigma}{1+\sigma}}\left(\mint_{B_{\rr/4}}\snr{D\vartheta}^{(1+\sigma)p_{2}(\rr)} \ dx\right)^{\frac{1}{1+\sigma}}\nonumber\\
\le &c\rr^{n+\frac{\tilde{\beta}\gamma_{1}\sigma}{1+\sigma}}\left(\mint_{B_{\rr/4}}\snr{D\omega}^{1+\sigma_{0}p(x)} \ dx\right)^{\frac{p_{2}(\rr)}{(1+\sigma_{0})p_{1}(\rr)}}\nonumber\\
\le &c\rr^{n+\frac{\tilde{\beta}\gamma_{1}\sigma}{1+\sigma}}\left(\mint_{B_{\rr/2}}(1+\snr{D\omega}^{2})^{p(x)/2}\right)^{\frac{p_{2}(\rr)-p_{1}(\rr)}{p_{1}(\rr)}}\mint_{B_{\rr}}(1+\snr{D\omega}^{2})^{p_{2}(\rr)/2} \ dx\nonumber\\
\le& c\rr^{\frac{\tilde{\beta}\gamma_{1}\sigma}{1+\sigma}}\mint_{B_{\rr}}(1+\snr{D\omega}^{2})^{p_{2}(\rr)/2} \ dx,
\end{flalign}
where $c=c(\texttt{data},\nr{(\snr{Du})}_{L^{1}(\tilde{\Omega})},[u]_{0,\tilde{\beta};\tilde{\Omega}})$. Collecting estimates \eqref{gr3}-\eqref{gr8} we can conclude that
\begin{flalign*}
\int_{B_{\rr/4}}(\snr{D\omega}^{2}+\snr{D\vartheta}^{2})^{\frac{p_{2}(\rr)-2}{2}} \ dx \le c\left(\rr^{\frac{\tilde{\beta}\gamma_{1}\sigma}{1+\sigma}}+\rr^{\bar{\kappa}_{1}}+\rr^{\tilde{\beta}}\right)\int_{B_{\rr}}(1+\snr{D\omega}^{2})^{p_{2}(\rr)/2} \ dx,
\end{flalign*}
with $c=c(\texttt{data},\nr{(\snr{Du})}_{L^{1}(\tilde{\Omega})},[u]_{0,\tilde{\beta};\tilde{\Omega}})$. Manipulating the content of the previous display as we did in \emph{Step 1}, estimates \eqref{36}-\eqref{37} we can conclude that
\begin{flalign}\label{gr9}
\int_{B_{\rr/4}}\snr{D\omega-D\vartheta}^{p_{2}(\rr)} \ dx \le c\left(\rr^{\frac{\tilde{\beta}\gamma_{1}\sigma}{2(1+\sigma)}}+\rr^{\frac{\bar{\kappa}_{1}}{2}}+\rr^{\frac{\tilde{\beta}}{2}}\right)\int_{B_{\rr}}(1+\snr{D\omega}^{2})^{p_{2}(\rr)/2} \ dx.
\end{flalign}
Recalling also that, by \eqref{beta}, $\tilde{\beta}\ge 1/2$, we can rewrite \eqref{gr9} as
\begin{flalign}\label{gr10}
\int_{B_{\rr/4}}\snr{D\omega-D\vartheta}^{p_{2}(\rr)} \ dx \le c\rr^{\kappa_{2}}\int_{B_{\rr}}(1+\snr{D\omega}^{2})^{p_{2}(\rr)/2} \ dx,
\end{flalign}
with $\kappa_{2}:=\frac{1}{2}\min\left\{\frac{1}{2},\bar{\kappa}_{1},\frac{\gamma_{1}\sigma}{2(1+\sigma)}\right\}$. Averaging in \eqref{gr10} and using \eqref{beta} again, we readily see that
\begin{flalign}
\mint_{B_{\rr/4}}\snr{D\omega-D\vartheta}^{p_{2}(\rr)} \ dx\stackrel{\eqref{morom}_{3}}{\le} c\rr^{\hat{\kappa}},\label{gr11}
\end{flalign}
for $c=c(\texttt{data},\nr{(\snr{Du}^{p(\cdot)})}_{L^{1}(\tilde{\Omega})},[u]_{0,\tilde{\beta};\tilde{\Omega}})$. Here $\hat{\kappa}:=\kappa_{2}-\gamma_{2}(1-\tilde{\beta})\ge \kappa_{2}/2>0$. Now fix any $0<\varsigma<\varrho/8$ and notice that, being $\vartheta$ a solution to \eqref{fudp}, the decay estimate $\eqref{140}_{2}$ holds true. So we estimate
\begin{flalign}\label{s612}
\mint_{B_{\varsigma}}\snr{D\omega-(D\omega)_{\varsigma}}^{p_{2}(\rr)}\ dx\le &c\left\{(\rr/\varsigma)^{n}\mint_{B_{\rr/4}}\snr{D\omega-D\vartheta}^{p_{2}(\rr)}\ dx+\mint_{B_{\varsigma}}\snr{D\vartheta-(D\vartheta)_{\varsigma}}^{p_{2}(\rr)} \ dx\right\}\nonumber \\
\stackleq{gr11}&c\left\{(\rr/\varsigma)^{n}\rr^{\hat{\kappa}}+(\varsigma/\rr)^{\mu p_{2}(\rr)}\mint_{B_{\rr/4}}\snr{D\omega}^{p_{2}(\rr)} \ dx\right\}\nonumber \\
\stackleq{morom}&c\left\{(\rr/\varsigma)^{n}\rr^{\hat{\kappa}}+(\varsigma/\rr)^{\mu p_{2}(\rr)}\rr^{-\gamma_{2}(1-\beta)}\right\},
\end{flalign}
with $\beta \in (0,1)$ still to be fixed and $c=c(\texttt{data},\nr{(\snr{Du})^{p(\cdot)}}_{L^{1}(\tilde{\Omega})},[u]_{0,\tilde{\beta};\tilde{\Omega}},\beta)$. Set $\beta:=1-\frac{\mu\gamma_{1}\hat{\kappa}}{2n\gamma_{2}}$ in \eqref{s612} and pick $\varsigma=\rr^{1+a}/2$ with $a:=\frac{\hat{\kappa}(2n+\mu\gamma_{1})}{2n(n+\mu p_{2}(\rr))}$. In these terms, \eqref{s612} reads as
\begin{flalign*}
\mint_{B_{\varsigma}}\snr{D\omega-(D\omega)_{\varsigma}}^{p_{2}(\rr)} \ dx \le &c\left\{\varsigma^{\frac{-an+\hat{\kappa}}{1+a}}+\varsigma^{\frac{a\mu p_{2}(\rr)-\gamma_{2}(1-\beta)}{1+a}}\right\}\le c\varsigma^{\frac{n\mu\hat{\kappa}\gamma_{1}}{2n(n+\mu\gamma_{2})+\hat{\kappa}(2n+\mu\gamma_{1})}}=c\varsigma^{\beta_{0}\gamma_{2}},
\end{flalign*}
where we also denoted
\begin{flalign}\label{beta0}
\beta_{0}:=\frac{n\mu\hat{\kappa}\gamma_{1}}{2n\gamma_{2}(n+\mu\gamma_{2})+\hat{\kappa}\gamma_{2}(2n+\mu\gamma_{1})}.
\end{flalign}
From the content of the previous display and H\"older inequality we finally get
\begin{flalign*}
\left(\mint_{B_{\varsigma}}\snr{D\omega-(D\omega)_{\varsigma}}^{\gamma_{1}} \ dx\right)^{\frac{1}{\gamma_{1}}}\le \left(\mint_{B_{\varsigma}}\snr{D\omega-(D\omega)_{\varsigma}}^{p_{2}(\rr)} \ dx \right)^{\frac{1}{p_{2}(\rr)}}\le c\varsigma^{\beta_{0}},
\end{flalign*}
thus
\begin{flalign*}
\mint_{B_{\varsigma}}\snr{D\omega-(D\omega)_{\varsigma}}^{\gamma_{1}} \ dx \le c\varsigma^{\gamma_{1}\beta_{0}},
\end{flalign*}
with $c=c(\texttt{data},\nr{(\snr{Du})^{p(\cdot)}}_{L^{1}(\tilde{\Omega})},[u]_{0,\tilde{\beta},\tilde{\Omega}})$, so, after covering, we can conclude that $D\omega\in C^{0,\beta_{0}}_{\mathrm{loc}}(\tilde{\Omega},\mathbb{R}^{m\times n})$ because of Morrey's growth theorem. By \eqref{beta0}, it is evident that $\beta_{0}=\beta_{0}(\texttt{data})$ does not depend on $\tilde{\Omega}$, thus $\eqref{f}_{2,3}$ and a standard covening argument render that $Du \in C^{0,\beta_{0}}_{\mathrm{loc}}(\Omega_{0},\mathbb{R}^{N\times n})$.\\\\
\emph{Step 7: the case $p(\cdot)>n-\delta_{0}/2$}. As mentioned in \emph{Step 1}, $u\in C^{0,\beta'}(\Omega^{+},\m)$, with $\beta':=\frac{\delta_{0}}{4n+\delta_{0}}$, so we no longer need to impose a smallness condition like \eqref{small}. Being $p(\cdot)$ continuous, $\Omega^{+}$ is open, so we can fix a ball $B_{\tilde{R}_{*}}\equiv B_{\tilde{R}_{*}}(x_{0})\Subset \Omega^{+}$ with $\tilde{R}_{*}$ satisfying \eqref{18}. Let $\sigma_{0}$ be as in \eqref{sigma0}, so \eqref{kom} is matched on all balls $B_{4\rr}\subset B_{R_{*}}\subset B_{\tilde{R}_{*}}$, where the size of $R_{*}\le \tilde{R}_{*}/2$ will be specified later on. As we did in \emph{Step 6}, we fix open subsets $\tilde{\Omega}\Subset\Omega'\Subset \Omega^{+}$ and cover $\tilde{\Omega}$ with a finite number of balls contained inside $\Omega'$ whose size and number will now depend on $\m$, on $[u]_{0,\beta';\tilde{\Omega}}$ and on $\diam(\Omega')$, having images contained in small coordinate neighborhoods of $\m$. Again we can find $\omega\in W^{1,p(\cdot)}(\tilde{\Omega},\mathbb{R}^{m})\cap C^{0,\beta'}(\tilde{\Omega},\mathbb{R}^{m})$, unconstrained local minimizer of the variational integral \eqref{uvp} with integrand $H(\cdot)$ matching \eqref{assh}, such that $\snr{\omega}\le 1$, $u=f(\omega)$ where $f$ is as in \eqref{f}. Our goal is to show the validity of a Morrey decay estimate like $\eqref{morom}_{3}$. To do so, fix $B_{4\rr}\Subset B_{R_{*}}$ and let $\vartheta \in W^{1,p_{2}(\rr)}(B_{\rr/4},\mathbb{R}^{m})$ be a solution to the frozen Dirichlet problem \eqref{fudp}. Notice that the estimates obtained in \emph{Step 6} till \eqref{gr9} do not require any specific value of $\beta$, therefore, by \eqref{gr9} with $\tilde{\beta}$ replaced by $\beta'$ we immediately have
\begin{flalign}\label{gr12}
\int_{B_{\rr/4}}\snr{D\omega-D\vartheta}^{p_{2}(\rr)}\ dx \le c\rr^{\kappa'}\int_{B_{\rr}}(1+\snr{D\omega}^{2})^{p_{2}(\rr)} \ dx,
\end{flalign}
with $c=c(\texttt{data},\nr{(\snr{Du}^{p(\cdot)})}_{L^{1}(\tilde{\Omega})},[u]_{0,\beta';\tilde{\Omega}})$ and $\kappa':=\frac{1}{2}\min\left\{\frac{\beta'\gamma_{1}\sigma}{1+\sigma},\bar{\kappa}_{1},\beta'\right\}$. Now fix $\tau \in \left(0,\frac{1}{8}\right)$ and recall that, being $\vartheta$ a solution of \eqref{fudp}, inequality $\eqref{140}_{1}$ holds for all $B_{\varsigma_{1}}\subset B_{\varsigma_{2}}\subset B_{\rr/4}$. Adopting the same terminology appearing in \emph{Step 4}, clearly with $\omega$ instead of $u$, we readily have
\begin{flalign}\label{s76}
\psi(\tau \rr)\stackleq{122} &\phi(\tau \rr, \rr)\le c\left\{(\tau\rr)^{p_{2}(\rr)-n}\left[(\tau\rr)^{n}+\int_{B_{\tau \rr}}\snr{D\omega-D\vartheta}^{p_{2}(\rr)} \ dx+\int_{B_{\tau \rr}}\snr{D\vartheta}^{p_{2}(\rr)} \ dx\right]\right\}^{\frac{1}{p_{2}(\rr)}}\nonumber \\
\stackleq{gr12} &c\left\{\tau^{p_{2}(\rr)-n}\left[\tau^{n}+\rr^{\kappa'}\right]\rr^{p_{2}(\rr)-n}\int_{B_{\rr}}(1+\snr{D\omega}^{2})^{p_{2}(\rr)} \ dx\right\}^{\frac{1}{p_{2}(\rr)}}\nonumber \\
\le &\tau^{\beta}\left[c\tau^{1-\beta}+c\rr^{\frac{\kappa'}{\gamma_{2}}}\tau^{-\beta-\frac{\delta_{0}}{2n-\delta_{0}}}\right]\psi(\rr),
\end{flalign}
where we also used $p(\cdot)>n-\frac{\delta_{0}}{2}$. Here $\beta\in (0,1)$ is arbitrary and $c=c(\texttt{data},\nr{(\snr{Du})^{p(\cdot)}}_{L^{1}(\tilde{\Omega})},[u]_{0,\beta';\tilde{\Omega}},\beta)$. Choosing in \eqref{s76} $\tau\le (2c)^{-1/(1-\beta)}$ and $R_{*}\le c^{-\frac{\gamma_{2}}{\kappa'}} 2^{-\frac{\gamma_{2}}{\kappa'}}\tau^{\frac{2n\gamma_{2}}{(2n-\delta_{0})\kappa'}}$ we end up with $\psi(\tau\rr)\le \tau^{\beta}\psi(\rr)$, by remembering also that $\rr\le R_{*}$. Notice that our previous decisions fixed the following dependencies: $\tau=\tau(\texttt{data},\nr{(\snr{Du})^{p(\cdot)}}_{L^{1}(\tilde{\Omega})},[u]_{0,\beta';\tilde{\Omega}},\beta)$ and $R_{*}=R_{*}(\texttt{data},\nr{(\snr{Du})^{p(\cdot)}}_{L^{1}(\tilde{\Omega})},[u]_{0,\beta';\tilde{\Omega}},\beta)$. By induction, it is easy to see that for any integer $j$ there holds
\begin{flalign}\label{s77}
\psi(\tau^{j}\rr)\le \tau^{j\beta}\psi(\rr).
\end{flalign}
Now, if $0<\varsigma<\rr/8$, there exists an integer $j\ge 1$ such that $\tau^{j+1}\rr<\varsigma\le \tau^{j}\rr$. Therefore, proceeding as we did for \eqref{46}, using \eqref{s77} we get
\begin{flalign}\label{s78}
\psi(\varsigma)\le \tau^{1-\frac{n}{\gamma_{1}}}\tau^{j\beta}\psi(\rr)\le c(\varsigma/\rr)^{\beta}\psi(\rr),
\end{flalign}
with $c=c(\texttt{data},\nr{(\snr{Du})^{p(\cdot)}}_{L^{1}(\tilde{\Omega})},[u]_{0,\beta';\tilde{\Omega}},\beta)$. This is the estimate we were looking for. In fact, as in \emph{Step 4}, \eqref{n2} we can extend \eqref{s78} to the full range $0<\varsigma<\rr$ and, proceeding as in estimates \eqref{n0}-\eqref{n1} we can get rid of the restriction $s\le R_{*}$; as already mentioned, we shall only retain $s\le \tilde{R}_{*}/2$. Furthermore, it directly implies that
\begin{flalign*}
\varsigma^{\gamma_{1}-n-\beta\gamma_{1}}\int_{B_{\varsigma}}\snr{D\omega}^{\gamma_{1}} \ dx\le& c\tilde{R}_{*}^{\gamma_{1}(1-\beta)}\left(\mint_{B_{\tilde{R}_{*}/2}}(1+\snr{D\omega}^{2})^{p_{2}(\tilde{R}_{*}/2)/2} \ dx\right)^{\frac{\gamma_{1}}{p_{2}(\tilde{R}_{*}/2)}}\nonumber \\
\le &c\tilde{R}_{*}^{\gamma_{1}(1-\beta)}\left(\mint_{B_{\tilde{R}_{*}/2}}(1+\snr{D\omega}^{2})^{(1+\sigma_{0})p(x)} \ dx\right)^{\frac{\gamma_{1}}{(1+\sigma_{0})p_{1}(\tilde{R}_{*}/2)}}\nonumber \\
\le& c\tilde{R}_{*}^{\gamma_{1}(1-\beta)}\left(\mint_{B_{\tilde{R}_{*}}}(1+\snr{D\omega}^{2})^{p(x)/2} \ dx\right)^{\frac{\gamma_{1}}{p_{1}(\tilde{R}_{*}/2)}}\le c,
\end{flalign*}
for $c=c(\texttt{data},\nr{(\snr{Du})^{p(\cdot)}}_{L^{1}(\tilde{\Omega})},[u]_{0,\beta';\tilde{\Omega}},\beta)$ and therefore, being $\beta \in (0,1)$ arbitrary, by Morrey's growth theorem and a standard covering argument, we can conclude that $\omega \in C^{0,\beta}_{\mathrm{loc}}(\Omega^{+},\mathbb{R}^{m})$ for any $\beta \in (0,1)$. Now, for all $B_{4\varsigma}\Subset \Omega^{+}$ such that $0<\varsigma\le \tilde{R}_{*}/2$, by Lemmas \ref{inngeh}, \ref{cacc} and \ref{L0} (\textit{ii}.), we obtain
\begin{flalign}\label{s79}
\mint_{B_{\varsigma}}\snr{D\omega}^{p_{2}(\varsigma)} \ dx \le &c\left(\mint_{B_{2\varsigma}}(1+\snr{D\omega}^{2})^{p(x)/2} \ dx\right)^{\frac{p_{2}(\varsigma)-p_{1}(\varsigma)}{p_{1}(\varsigma)}}\mint_{B_{2\varsigma}}(1+\snr{D\omega}^{2})^{p(x)/2} \ dx\nonumber \\
\le &c\mint_{B_{4\varsigma}}1+\left |\frac{\omega-(\omega)_{4\varsigma}}{\varsigma}\right |^{p(x)} \ dx\le c\varsigma^{-\gamma_{2}(1-\beta)},
\end{flalign}
where $c=c(\texttt{data},\nr{(\snr{Du})^{p(\cdot)}}_{L^{1}(\tilde{\Omega})},[u]_{0,\beta';\tilde{\Omega}},\beta)$. Once \eqref{s79} is available, we can conclude as in \emph{Step 6}.
\section{Dimension reduction}
In this section we obtain a further reduction of the dimension of the singular set of $p(x)$-harmonic maps, for $p(\cdot)\ge 2$ Lipschitz continuous, thus improving, at least in this case, the result given in Theorem \ref{T0}, \emph{Step 5}. 
\subsection{Compactness of minimizers and Monotonicity formula}
The proof of Theorem \ref{T1} essentially needs two components to be carried out. The first is the compactness of sequences of minimizers of \eqref{cvp} under uniform assumptions, while the second is the monotonicity along solutions to \eqref{cvp} of a certain quantity strictly related to the $p(x)$-energy. Those arguments are quite classical, see e. g. \cite{giamar,harkinlin,tac}.
\begin{lemma}[Compactness]\label{com}
Let $(k_{j})_{j \in \N}, (p_{j})_{j \in \N}$ be two sequences of $\alpha$-H\"older continuous functions, $\alpha \in (0,1]$, satisfying
\begin{flalign}\label{com1}
\begin{cases}
\ \sup_{j \in \N}[k_{j}]_{0,\alpha}<c_{k}\\
\ \lambda\le k_{j}(x)\le \Lambda \ \mathrm{for \ all \ }x \in B_{1}\\
\ \nr{k_{j}-k}_{L^{\infty}(B_{1})}\to 0, \ k(\cdot) \in C^{0,\alpha}(B_{1})
\end{cases}\quad \mbox{and}\quad \begin{cases}
\ \sup_{j \in \N}[p]_{0,\alpha}<c_{p}\\
\ p_{j}(x)\ge \gamma_{1}>1 \ \mathrm{for \ all \ }x \in B_{1}, j \in \N\\
\ \nr{p_{j}-p_{0}}_{L^{\infty}(B_{1})}\to 0,\ p_{0}\ge \gamma_{1}>1 \ \mathrm{constant},
\end{cases}
\end{flalign}
respectively. For each $j \in \N$, let $u_{j}\in W^{1,p_{j}(\cdot)}(B_{1},\m)$ be a constrained local minimizer of
\begin{flalign*}
\mathcal{E}_{j}(w,B_{1}):=\int_{B_{1}}k_{j}(x)\snr{Dw}^{p_{j}(x)} \ dx,
\end{flalign*}
where $\m$ is as in $(\mathrm{\m 1})$-$(\mathrm{\m 2})$. Then, there exists a subsequence, still denoted by $(u_{j})_{j \in  \N}$, such that 
\begin{flalign}\label{com3}
u_{j}\rightharpoonup v \ \ \mathrm{weakly \ in \ }W^{1,(1+\tilde{\sigma})p_{0}}(B_{r},\m)
\end{flalign}
for some $\tilde{\sigma}>0$ and any $r \in (0,1)$ and $v$ is a constrained local minimizer of the functional
\begin{flalign*}
\mathcal{E}_{0}(w,B_{1}):=\int_{B_{1}}k(x)\snr{Dw}^{p_{0}} \ dx.
\end{flalign*}
Moreover, $\mathcal{E}_{j}(u_{j},B_{r})\to \mathcal{E}_{0}(v,B_{r})$ for all $r\in (0,1).$ Finally, if $x_{j}$ is a singular point of $u_{j}$ and $x_{j}\to \bar{x}$, then $\bar{x}$ is a singular point for $v$.
\end{lemma}
\begin{proof}
The proof is divided into three steps.\\\\
\emph{Step 1: weak $W^{(1+\tilde{\sigma})p_{0}}$-convergence.} Since $\m$ is compact, $\sup_{j\in \N}\nr{u_{j}}_{L^{\infty}(B_{1})}\le c(\m)$,
and given that $\gamma_{1}>1$, we obtain, up to (non relabelled) subsequences,
\begin{flalign}\label{139}
u_{j}\rightharpoonup v \ \mbox{weakly in }L^{\gamma_{1}}(B_{1},\mathbb{R}^{N}).
\end{flalign}
Moreover, being the bounds in \eqref{com1} uniform in $j \in \N$, Lemma \ref{inngeh} and Corollary \ref{C0} (and Lemma \ref{bougeh} for the associated frozen problems) hold for all the $\mathcal{E}_{j}$'s with parameters independent of $j$. By Lemma \ref{inngeh}, we know that $(u_{j})_{j \in \N}\subset W^{1,(1+\delta)p(\cdot)}(B_{1},\m)$ for all $\delta \in (0,\tilde{\delta}_{0})$. Let $\delta_{2}:=\frac{1}{4}\min\{\tilde{\sigma}_{0},\tilde{\delta}_{0}\}$, where $\tilde{\sigma}_{0}$ is the higher integrability threshold given by Lemma \ref{bougeh} and pick any $\delta \in (0,\delta_{2})$. Because of the uniform convergence of the $p_{j}$'s to the constant $p_{0}$, taking $j$ sufficiently large we can find positive constants $\gamma_{1}\le q_{1}\le q_{2}\le \gamma_{2}$ such that
\begin{flalign}\label{72}
1<q_{1}\le p_{j}(\cdot)\le q_{2}<\infty \ \mbox{on} \ B_{1}, \quad q_{2}\left(1+\frac{\delta}{2}\right)\le q_{1}(1+\delta), \quad q_{2}\le p_{0}\left(1+\frac{\delta}{2}\right).
\end{flalign}
For any $B_{\rr}(x_{0})\equiv B_{\rr}\subset B_{1}$, Corollary \ref{C0} yields that
\begin{flalign*}
\int_{B_{\rr/4}}\snr{Du_{j}}^{(1+\delta)p_{j}(x)} \ dx\le c
\end{flalign*}
for all $j \in \N$, with $c=c(\rr,c_{p},n,N,\m,\lambda,\Lambda,\gamma_{1},\gamma_{2},\alpha)$. This last estimate and $\eqref{72}_{1,2}$ imply that
\begin{flalign}\label{74}
\int_{B_{\rr/4}}\snr{Du_{j}}^{(1+\delta/2)q_{2}} \ dx \le c,
\end{flalign}
for $c=c(\rr,c_{p},n,N,\m,\lambda,\Lambda,\gamma_{1},\gamma_{2},\alpha)$. Now, for any fixed $r\in (0,1)$, we can cover $B_{r}\equiv B_{r}(0)$ by a finite number of balls $B_{(1-\rr)/4}(x_{0})$ with $x_{0} \in B_{r}$, use \eqref{74} on each ball and then sum them all to get
\begin{flalign}\label{75}
\int_{B_{r}}\snr{Du_{j}}^{(1+\delta/2)q_{2}} \ dx \le c
\end{flalign}
for large $j \in \N$. Here $c=c(r,c_{p},n,N,\m,\lambda,\Lambda,\gamma_{1},\gamma_{2},\alpha)$. From the compactness of $\m$ and \eqref{75}, we derive the uniform boundedness of the $u_{j}$'s in $W^{1,(1+\delta/2)q_{2}}(B_{r},\m)$, so, up to extract a (non relabelled) subsequence, we obtain that $u_{j}\rightharpoonup \bar{v}$ weakly in $W^{1,(1+\delta/2)q_{2}}(B_{r},\m)$, for some $\bar{v}\in W^{1,(1+\delta/2)q_{2}}(B_{r},\m)$. Anyway, by \eqref{139}, $\bar{v}(x)=v(x)$, $v(x)\in \m$ for a.e. $x\in B_{r}$ and, by Rellich's theorem,
\begin{flalign}
u_{j}\to v \ \ &\mbox{strongly in} \ \ L^{(1+\delta/2)q_{2}}(B_{r},\m),\label{77}\\
Du_{j}\to Dv \ \ &\mbox{weakly in}\ \ L^{(1+\delta/2)q_{2}}(B_{r},\mathbb{R}^{N\times n}).\label{79}
\end{flalign}
From $\eqref{72}_{1}$ and $\eqref{com1}_{2}$ we see that $q_{2}\ge p_{0}$, so \eqref{com3} is proved with $\tilde{\sigma}\equiv\delta/2$. In particular, the weak lower semicontinuity of the norm renders that
\begin{flalign}\label{78}
\int_{B_{r}}\snr{Dv}^{(1+\delta/2)q_{2}} \ dx \le c,
\end{flalign}
with $c=c(r,c_{p},n,N,\m,\lambda,\Lambda,\gamma_{1},\gamma_{2},\alpha)$.\\\\
\emph{Step 2: compactness.} We aim to show that $v$ is an $\m$-constrained local minimizer of $\mathcal{E}_{0}$. To do so, we first claim that
\begin{flalign}\label{66}
\mathcal{E}_{0}(v,B_{r})\le \liminf_{j\to \infty}\mathcal{E}_{j}(u_{j},B_{r}),
\end{flalign}
for all $r \in (0,1)$. Let us rewrite $\mathcal{E}_{j}(u_{j},B_{r})=\left(\mathcal{E}_{j}(u_{j},B_{r})-\mathcal{E}_{0}(u_{j},B_{r})\right)+\mathcal{E}_{0}(u_{j},B_{r})$. From \eqref{com3} and weak lower semicontinuity we have
\begin{flalign}\label{67}
\mathcal{E}_{0}(v,B_{r})\le \liminf_{j\to \infty}\mathcal{E}_{0}(u_{j},B_{r}).
\end{flalign}
On the other hand, from $\eqref{72}_{1},$ Lemma \ref{L0} (\textit{i.}) with $\varepsilon_{0}\equiv\delta/2$ and \eqref{com1} we have
\begin{flalign}\label{68}
\left | \ \mathcal{E}_{j}(u_{j},B_{r})-\mathcal{E}_{0}(u_{j},B_{r}) \ \right |\le &c\nr{p_{j}-p_{0}}_{L^{\infty}(B_{1})}\int_{B_{r}}1+\snr{Du_{j}}^{(1+\delta/2)q_{2}} \ dx\nonumber\\
&+\nr{k_{j}-k}_{L^{\infty}(B_{1})}\int_{B_{r}}\snr{Du_{j}}^{p_{0}} \ dx\nonumber \\
\stackleq{75}&c\left(\nr{p_{j}-p_{0}}_{L^{\infty}(B_{1})}+\nr{k_{j}-k}_{L^{\infty}(B_{1})}\right)\to0,
\end{flalign}
where $c=c(r,c_{p},n,N,\m,\lambda,\Lambda,\gamma_{1},\gamma_{2},\alpha)$. Combining \eqref{68} and \eqref{67} we obtain \eqref{66}.\\
Let $\tilde{v}\in W^{1,p_{0}}(B_{r},\m)$ be a solution to the Dirichlet problem
\begin{flalign*}
W^{1,p_{0}}_{v}(B_{r},\m)\ni w\mapsto \min \mathcal{E}_{0}(w,B_{r}),
\end{flalign*}
and extend it to be equal to $v$ outside $B_{r}$. In this way, $\tilde{v}\in W^{1,p_{0}}_{v}(B_{1},\m)$. Since we are assuming that $(p_{j})_{j\in \N}$ converges uniformly to $p_{0}$ on $B_{1}$, we can take $j \in \N$ so large that
\begin{flalign}\label{69}
\nr{p_{j}}_{L^{\infty}(B_{1})}\left(1+\frac{\delta}{4}\right)\le p_{0}\left(1+\frac{\delta}{2}\right)
\end{flalign}
holds. Moreover, by \eqref{78} and $\eqref{72}_{1}$ we have that $v \in W^{1,(1+\delta/2)q_{2}}(B_{r},\m)\subset W^{1,(1+\delta/2)p_{0}}(B_{r},\m)$, so, from Lemma \ref{bougeh} with $p\equiv p_{0}$ we obtain that $\tilde{v}\in W^{1,(1+\delta/2)p_{0}}(B_{r},\m)\subset W^{1,(1+\delta/4)p_{j}(\cdot)}(B_{r},\m)\cap W^{1,q_{2}}(B_{r},\m)$, where the last inclusion is due to \eqref{69} and $\eqref{72}_{3}$. From \eqref{69}, Lemma \ref{bougeh} and \eqref{78} we get
\begin{flalign}\label{80}
\int_{B_{r}}\snr{D\tilde{v}}^{(1+\delta/4)p_{j}(x)} \ dx\le \int_{B_{r}}1+\snr{D\tilde{v}}^{(1+\delta/2)p_{0}} \ dx\le c\int_{B_{r}}1+\snr{Dv}^{(1+\delta/2)p_{0}} \ dx\le c,
\end{flalign}
where $c=c(r,c_{p},n,N,\m,\lambda,\Lambda,\gamma_{1},\gamma_{2},\alpha)$. Let $\theta \in (0,1)$ be a small parameter to be fixed and $\eta$ a cut-off function with the following specifics
\begin{flalign}\label{cutoff}
\eta \in C^{1}_{c}(B_{r}), \quad \chi_{B_{r(1-\theta)}}\le \eta\le \chi_{B_{r}},\quad \snr{D\eta}\le (r\theta)^{-1}\ \ \mbox{on}\ \ A_{r\theta}.
\end{flalign}
In correspondence of such a choice of $\eta$, we define the comparison map $w_{j}:=(1-\eta)u_{j}+\eta\tilde{v}$ and notice that $\left. w_{j} \right |_{\partial B_{r(1-\theta)}}=\left. \tilde{v} \right |_{\partial B_{r(1-\theta)}}$ and $\left. w_{j} \right |_{\partial B_{r}}=\left. u_{j} \right |_{\partial B_{r}}$. So Lemma \ref{ext} applies to $w_{j}$ on $A_{r\theta}$ thus rendering a map $w_{j}'\in W^{1,p_{j}(\cdot)}(A_{r\theta},\m)$ such that
\begin{flalign}\label{comex}
\left. w_{j}' \right |_{\partial B_{r(1-\theta)}}=\left. \tilde{v} \right |_{\partial B_{r(1-\theta)}},\quad \left. w_{j}' \right |_{\partial B_{r}}=\left. u_{j} \right |_{\partial B_{r}}, \quad \int_{A_{r\theta}}\snr{Dw'_{j}}^{p_{j}(x)} \ dx\le c\int_{A_{r\theta}}\snr{Dw_{j}}^{p_{j}(x)} \ dx,
\end{flalign}
with $c=c(N,\m,\gamma_{2})$. Finally we define
\begin{flalign}\label{tf}
\tilde{w}_{j}:=\begin{cases} \ \tilde{v} \ \ &\mbox{on} \ B_{r(1-\theta)}\\
\ w'_{j} \ \ &\mbox{on} \ A_{r\theta},
\end{cases}
\end{flalign}
which, by $\eqref{comex}_{2,3}$ and \eqref{80} is an admissible competitor for $u_{j}$ on $B_{r}$. From the minimality of $u_{j}$ and $\eqref{comex}_{2,3}$ we have
\begin{flalign*}
\mathcal{E}_{j}(u_{j},B_{r})\le &\mathcal{E}_{j}(\tilde{w}_{j},B_{r})=\mathcal{E}_{j}(\tilde{v},B_{r(1-\theta)})+\mathcal{E}_{j}(w'_{j},A_{r\theta})\\
\le &\int_{B_{r}}k_{j}(x)\snr{D\tilde{v}}^{p_{j}(x)} \ dx +c\int_{A_{r\theta}}k_{j}(x)\snr{Dw_{j}}^{p_{j}(x)} \ dx:=\mbox{(I)}_{j}+\mbox{(II)}_{j},
\end{flalign*}
for $c=c(N,\m,\lambda,\Lambda,\gamma_{2})$. By \eqref{com1}, \eqref{69}, $\eqref{72}_{1}$, \eqref{78} and Lemma \ref{L0} (\textit{i.}) with $\varepsilon_{0}\equiv\delta/4$ we see that
\begin{flalign*}
\left |\ \int_{B_{r}}k_{j}(x)\snr{D\tilde{v}}^{p_{j}(x)} \ dx\right.&\left .-\int_{B_{r}}k(x)\snr{D\tilde{v}}^{p_{0}} \ dx \  \right |\\
\stackleq{80} &c\left(\nr{k_{j}-k}_{L^{\infty}(B_{1})}+\nr{p_{j}-p_{0}}_{L^{\infty}(B_{1})}\right)\int_{B_{r}}1+\snr{Dv}^{(1+\delta/2)p_{0}} \ dx\to 0,
\end{flalign*}
where $c=c(r,c_{p},n,N,\m,\lambda,\Lambda,\gamma_{1},\gamma_{2},\alpha)$. So we get that
\begin{flalign}\label{82}
\mbox{(I)}_{j}\to\mathcal{E}_{0}(\tilde{v},B_{r}).
\end{flalign}
Exploiting the very definition of the $w_{j}$'s and $\eqref{cutoff}_{3}$ we have
\begin{flalign*}
\mbox{(II)}_{j}\le &c\int_{A_{r\theta}}k_{j}(x)\left[\snr{Du_{j}}^{p_{j}(x)}+\snr{D\tilde{v}}^{p_{j}(x)}+\left |\frac{u_{j}-\tilde{v}}{r\theta} \right |^{p_{j}(x)}\right] \ dx\\
\le &c\int_{A_{r\theta}}k_{j}(x)\left[\snr{Du_{j}}^{p_{j}(x)}+\snr{D\tilde{v}}^{p_{j}(x)}+\left |\frac{u_{j}-v}{r\theta} \right |^{p_{j}(x)}+\left |\frac{\tilde{v}-v}{r\theta} \right |^{p_{j}(x)}\right] \ dx=:\mbox{(II)}_{j}^{1}+\mbox{(II)}_{j}^{2}+\mbox{(II)}_{j}^{3},
\end{flalign*}
with $c=c(r,c_{p},n,N,\m,\lambda,\Lambda,\gamma_{1},\gamma_{2},\alpha)$. Using $\eqref{72}_{1}$, \eqref{74}, \eqref{80} and $\eqref{com1}$ we get
\begin{flalign}\label{83}
\mbox{(II)}_{j}^{1}\le &\Lambda \int_{A_{r\theta}}\snr{Du_{j}}^{p_{j}(x)}+\snr{D\tilde{v}}^{p_{j}(x)} \ dx \le c,
\end{flalign}
where $c=c(r,c_{p},n,N,\m,\lambda,\Lambda,\gamma_{1},\gamma_{2},\alpha)$. By $\eqref{72}_{1}$ and \eqref{77}, a well known variation on Lebesgue's dominated convergence theorem allows concluding that
\begin{flalign}\label{84}
\mbox{(II)}_{j}^{2}\to0.
\end{flalign}
Finally, by \eqref{14} with $p(\cdot)\equiv p_{j}(\cdot)$ and \eqref{80} we obtain
\begin{flalign}\label{85}
\mbox{(II)}_{j}^{3}\le c\int_{A_{r\theta}}\snr{Dv-D\tilde{v}}^{p_{j}(x)} \ dx +c\snr{A_{r\theta}}\le c\int_{A_{r\theta}}1+\snr{Dv}^{(1+\delta/2)p_{0}} \ dx+c\snr{A_{r\theta}}\le c,
\end{flalign}
for $c=c(r,c_{p},n,N,\m,\lambda,\Lambda,\gamma_{1},\gamma_{2},\alpha)$. By the absolute continuity of Lebesgue's integral, \eqref{83} and \eqref{85}, given $\sigma>0$ we can always choose $\theta$ sufficiently small in such a way that
\begin{flalign}\label{86}
(\mathrm{II})_{j}^{1}+(\mathrm{II})_{j}^{3}\le \frac{\sigma}{2},
\end{flalign}
and, by \eqref{84}, $j$ large enough such that
\begin{flalign}\label{87}
(\mathrm{II})^{2}_{j}\le \frac{\sigma}{2}.
\end{flalign}
All in all, collecting \eqref{66}, \eqref{82}, \eqref{86} and \eqref{87} we can conclude that
\begin{flalign*}
\mathcal{E}_{0}(v,B_{r})\le \liminf_{j\to \infty}\mathcal{E}_{j}(u_{j},B_{r})\le\limsup_{j\to \infty}\mathcal{E}_{j}(u_{j},B_{r})\le  \limsup_{j\to \infty}\mathcal{E}_{j}(\tilde{v},B_{r})+\sigma=\mathcal{E}_{0}(\tilde{v},B_{r})+\sigma,
\end{flalign*}
so, by the arbitrariety of $\sigma$ and the minimality of $\tilde{v}$, we can conclude that $\mathcal{E}_{0}(\tilde{v},B_{r})=\mathcal{E}_{0}(v,B_{r})$. Being this true for any $r \in (0,1)$, $v$ is an $\m$-constrained local minimizer of $\mathcal{E}_{0}$ and, as a direct consequence of the last chain of inequalities, $\mathcal{E}_{j}(u_{j},B_{r})\to \mathcal{E}_{0}(v,B_{r})$.\\\\
\emph{Step 3: singular points.} Once we have the results contained in \emph{Steps 1}-\emph{2} by hand, the proof of \emph{Step 3} goes as the one in \cite[Lemma 3.1]{tac} and we shall omit it.
\end{proof}
We stress that Lemma \ref{com} holds with $p(\cdot)\ge \gamma_{1}>1$ H\"older continuous rather than Lipschitz. We need stronger assumptions only to prove a suitable monotonicity formula.

\begin{lemma}[Monotonicity formula]\label{mon}
Let $k(\cdot)\in C^{0,\alpha}(\Omega)$, $\alpha \in (0,1]$ be such that $k(0)=1$, $p(\cdot) \in Lip(\Omega)$ and $n> \gamma_{2}\ge p(x)\ge 2$ for all $x \in \Omega$. If $u \in W^{1,p(\cdot)}(\Omega,\m)$ is a constrained local minimizer of \eqref{cvp} on $B_{1}$, then for any $\gamma \in (0,1)$ there exist a positive $c=c(\texttt{data},\gamma)$ and $T\in (0,1)$ such that for all $0<r<R<T$, we have
\begin{flalign*}
\int_{\partial B_{1}}\snr{u(Rx)-u(rx)}^{p_{2}(r)} \ d\mathcal{H}^{n-1}(x)\le cr^{p_{2}(r)-p_{2}(R)}\left(\log \frac{R}{r}\right)^{p_{2}(r)-1}\left((\Phi(R)-\Phi(r))+(R^{\gamma}-r^{\gamma})\right),
\end{flalign*}
where 
\begin{flalign*}
\Phi(t)=t^{p_{2}(t)-n}\exp({At^{\alpha}})\int_{B_{t}}k(x)\snr{Du}^{p_{2}(t)} \ dx,
\end{flalign*}
with $A=A(n,[k]_{0,\alpha},[p]_{0,1},\alpha)>0$.
\end{lemma}
\begin{proof}
The proof is actually the same as the one given in \cite[Lemma 4.1]{tac}. There is only one small detail to change: the map $v$ introduced during the proof of Lemma 4.1 to obtain estimate $(4.17)$ must be replaced by a solution to the Dirichlet problem
\begin{flalign*}
W^{1,p_{2}(t)}_{u}(B_{t},\m)\ni w\mapsto \inf \int_{B_{t}}k(x)\snr{Dw}^{p_{2}(t)} \ dx.
\end{flalign*}
The rest stays unchanged.
\end{proof}
\subsection{Proof of Theorem \ref{T1}}
Combining the compactness Lemma \ref{com} and the monotonicity formula obtained in Lemma \ref{mon}, we are ready to prove Theorem \ref{T1}. If $\Omega^{+}$ is as in \eqref{om1}, then $u \in W^{1,n+\delta_{0}/4}(\Omega^{+},\m)$. So, by Morrey's embedding theorem, $u \in C^{0,\beta'}(\Omega^{+},\m)$ for $\beta':=\delta_{0}/(4n+\delta_{0})$ and, by \emph{Step 7} of Theorem \ref{T0}, we can conclude that $Du$ is locally $\beta_{0}$-H\"older continuous on $\Omega^{+}$ for some $\beta_{0}\in (0,1)$. This observation shows that, to prove Theorem \ref{T1} it is enough to assume that $\gamma_{2}<n$, and this condition assures the applicability of Lemma \ref{mon}.\\\\
\emph{Case 1: $n\le [\gamma_{1}]+1$.} Since $t\mapsto \Phi(t)$ can be seen as a difference between an increasing function of $t$ and $c t^{\gamma}$ for some $\gamma\in (0,1)$ and a positive constant $c$, it admits a finite limit as $t\to 0$. Assume that $u$ has a singular point at $\bar{x}=0$ which is not isolated. Then we can find a sequence of singular points $(x_{j})_{j\in \N}$ such that $x_{j}\to0$. Setting $R_{j}:=2\snr{x_{j}}<T<1$ we see that, for any $j$ the rescaled function $u_{j}(x):=u(R_{j}x)$ is a constrained local minimizer of the functional
\begin{flalign*}
\mathcal{E}_{j}(w,B_{1}):=\int_{B_{1}}R_{j}^{p(0)-p_{j}(x)}\snr{Dw}^{p_{j}(x)} \ dx, \quad p_{j}(x):=p(R_{j}x)
\end{flalign*}
and each $u_{j}$ has a singular point $y_{j}:=R^{-1}_{j}x_{j}$ with $\snr{y_{j}}=1/2$. Now we notice that the sequences $(R_{j}^{p(0)-p_{j}(\cdot)})_{j \in \N}$ and $(p_{j}(\cdot))_{j \in \N}$ satisfy \eqref{com1}, so by Lemma \ref{com} we get, up to extract a subsequence that the $u_{j}$'s $L^{2}$-weakly converge to a function $v$, constrained local minimizer of $\mathcal{E}_{0}(w,B_{1}):=\int_{B_{1}}\snr{Dw}^{p(0)} \ dx$ and that the $y_{j}$'s converge to $\bar{y}$, singular point of $v$ with $\snr{\bar{y}}=1/2$. Now pick two constants $0<\lambda<\mu<1$ and apply Lemma \ref{mon} with $r\equiv\lambda R_{j}$ and $R\equiv\mu R_{j}$ to get
\begin{flalign}\label{123}
\int_{\partial B_{1}}&\snr{u_{j}(\mu x)-u_{j}(\lambda x)}^{p_{2}(\lambda R_{j})}d\mathcal{H}^{n-1}(x)=\int_{\partial B_{1}}\snr{u(\mu R_{j} x)-u(\lambda R_{j} x)}^{p_{2}(\lambda R_{j})}d\mathcal{H}^{n-1}(x)\nonumber \\
\le &c(\lambda R_{j})^{p_{2}(\lambda R_{j})-p_{2}(\mu R_{j})}\left(\log(\mu/\lambda)\right)^{p_{2}(\lambda R_{j})-1}\left((\Phi(\mu R_{j})-\Phi(\lambda R_{j}))+(\mu^{\gamma}-\lambda^{\gamma})R_{j}^{\gamma}\right)\to 0.
\end{flalign}
Moreover, Lemma \ref{com} also says that $u_{j}\to v$ in $L^{(1+\tilde{\sigma})p(0)}(B_{r},\m)$ for all $r\in (0,1)$ and this leads to
\begin{flalign}\label{124}
\snr{u_{j}(\mu x)-u_{j}(\lambda x)}^{p_{2}(\lambda R_{j})}\to \snr{v(\mu x)-v(\lambda x)}^{p(0)} \ \ \mbox{a.e. in }B_{1}.
\end{flalign}
Finally, the compactness of $\m$ renders the $u_{j}$'s uniformly bounded, so, by the dominated convergence theorem, \eqref{123} and \eqref{124} we deduce that
\begin{flalign*}
\int_{\partial B_{1}}\snr{v(\mu x)-v(\lambda x)}^{p(0)}d\mathcal{H}^{n-1}(x)=0,
\end{flalign*}
for a.e. $\lambda$ and $\mu$. This means that $v$ is homogeneous of degree $0$, so the whole segment joining $\bar{x}$ and $\bar{y}$ is made of singular points of $v$, but, since we are assuming $n\le [\gamma_{1}]+1\le [p(0)]+1$, we obtain a contradiction to \cite[Theorem 4.5]{harlin}, which states that, under these conditions, $v$ can have only isolated singularities.\\\\
\emph{Case 2: $n>[\gamma_{1}]+1$.} Let us assume that for some $l>0$, $\mathcal{H}^{l}(\Sigma_{0}(u))>0$. Then, by blowing up, we obtain a constrained local minimizer $v$ of $\mathcal{E}_{0}$ with $\mathcal{H}^{l}(\Sigma_{0}(v))>0$, (see \cite{giamar}, Chapter 10). On the other hand, by \cite[Theorem 4.5]{harlin}, $l<n-[p(0)]-1\le n-[\gamma_{1}]-1$ and this concludes the proof.\\\\
\textbf{Acknowledgments.} The author would like to thank the referees for the patient they had in reading the originally submitted manuscript and for the their most useful remarks, that eventually led to an improved version of the paper.\\
This work was supported by the \emph{Engineering and Physical Sciences Research Council} $[EP/L015811/1]$.


\begin{thebibliography}{50}
\bibitem{acefus} E. Acerbi, N. Fusco, Local regularity for minimizers of nonconvex integrals. Ann. Sc. Normale di Pisa, Cl. Sci., $4^{e}$ s\'erie, tome 16, nr. 4, p. 603-636, (1989).
\bibitem{acefus1} E. Acerbi, N. Fusco, Regularity for minimizers of non-quadratic functionals: the case $1<p<2$. J. Math. Anal. Appl. 140, 115-135, (1989).
\bibitem{acemin} E. Acerbi, G. Mingione, Regularity results for electrorheological fluids: the stationary case. C. R. Math. Acad. Sci. Paris 334 817-822, (2002).


\bibitem{barcolmin}P. Baroni, M. Colombo, G. Mingione, Regularity for General Functionals with Double Phase. Calc. Var. \& PDE 57:62, (2018).
\bibitem{carkrispas} M. Carozza, J. Kristensen, A. Passarelli di Napoli, Regularity of minimizers of autonomous convex variational integrals. Ann. Sc. Norm. Super. Pisa Cl. Sci. vol. XIII, issue 5, (2014).

\bibitem{colmin} M. Colombo, G. Mingione, Regularity for Double Phase Variational Problems. Arch. Rational Mech. Anal. 215 (2015), 443-496.

\bibitem{colmin2} M. Colombo, G. Mingione, Calder\'on-Zygmund estimates and non-uniformly elliptic operators. {\em J. Funct. Anal.} 270, 1416--1478, (2016). 
\bibitem{cosmin} A. Coscia, G. Mingione, H\"older continuity of the gradient of $p(x)$-harmonic mappings. C. R. Acad. Sci. Paris S\'er. I Math, 328 (4): 363-368, (1999).
\bibitem{me} C. De Filippis, Higher integrability for constrained minimizers of integral functionals with $(p,q)$-growth in low dimension. Nonlinear Analysis 170, 1-20, (2018).
\bibitem{demi} C. De Filippis, G. Mingione, Manifold constrained non-uniformly elliptic problems. Preprint, (2018).
 
\bibitem{diestrver1} L. Diening, B. Stroffolini, A. Verde, Everywhere regularity of functionals with $\varphi$-growth. Manuscripta Math. 129, 449-481, (2009). 
 
\bibitem{diestrver} L. Diening, B. Stroffolini, A. Verde, The $\varphi$-harmonic approximation and the regularity of $\varphi$-harmonic maps. J. Differential Equations 253, 1943-1958, (2012).
\bibitem{dotleomus} A. D'Ottavio, F. Leonetti, C. Musciano, Maximum principle for vector-valued mappings minimizing variational integrals, Atti
Semin. Mat. Fis. Univ. Modena 46 (Suppl.), 677-683, dedicated to Prof. C. Vinti (Italy, Perugia 1996), (1998).
\bibitem{ele} M. Eleuteri, H\"older continuity results for a class of functionals with non standard growth. Bollettino U. M. I., pages 129-157, 2004.

\bibitem {sharp}
L. Esposito, F. Leonetti, G. Mingione, Sharp regularity for functionals with $(p,q)$ growth. 
J.~Differential Equations 204, 5--55, (2004).
\bibitem{evagar} L. C. Evans, R. F. Gariepy, Blow-up, compactness and partial regularity in the calculus of variations. Indiana Univ. Math. J., Vol. 36, pp. 361-371, (1987).
\bibitem{fed} H. Federer, Curvature measures. Trans. Amer. Math. Soc., 93, 418-491, (1959).
\bibitem{fushut} N. Fusco, J. Hutchinson, Partial regularity for minimisers of certain functionals having nonquadratic growth. Ann. Mat. Pura Appl. 155(4), 1-24, (1989).
\bibitem{gia} M. Giaquinta, Multiple Integrals in the Calculus of Variations and Nonlinear Elliptic Systems. Princeton Univ. Press, Princeton, (1983).
\bibitem{giagiudiff} M. Giaquinta, E. Giusti, Differentiability of minima of non-differentiable functionals. Inventiones Math. 72, 285-298, (1983).
\bibitem{giagiu1} M. Giaquinta, E. Giusti, On the regularity of the minima of certain variational integrals. Acta Math., Vol. 148, 31-46, (1982). 
\bibitem{giagiu} M. Giaquinta, E. Giusti, The singular set of the minima of certain quadratic functionals. Ann. Sc. Norm. Sup. Pisa 9, 45-55, (1984).
\bibitem{giamar} M. Giaquinta, L. Martinazzi, An Introduction to the Regularity Theory for Elliptic Systems, Harmonic Maps and Minimal Graphs. Edizioni della Normale, (2012).
\bibitem{giamod} M. Giaquinta, G. Modica, Remarks on the regularity of the minimizers of certain degenerate functionals. Manuscripta Math. 57, 55-99, (1986).
\bibitem{giamodsou} M. Giaquinta, G. Modica, J. Sou\u cek, Cartesian Currents in the Calculus of Variations II. [Results in Mathematics and Related Areas. 3rd Series. A Series of Modern Surveys in Mathematics], vol. 38, Springer-Verlag, Berlin, 1998, Variational integrals. MR 1645082 (2000b:49001b)
\bibitem{giu} E. Giusti, Direct Methods in the calculus of variations. World Scientific Publishing Co., Inc., River Edge, 2003.
\bibitem{harkinlin} R. Hardt, D. Kinderlehrer, F.-H. Lin, Stable defects of minimizers of constrained variational principles. Ann. Inst. H. Poincar\'e Anal. Non Lin\'eaire 5, 297-322, (1988).
\bibitem{harlin} R. Hardt, F.-H. Lin, Mappings Minimizing the $L^{p}$ Norm of the Gradient. Comm. on Pure and Appl. Math., Vol. XL, 555-588, (1987).
\bibitem{harhaskle} P. Harjulehto, P. H\"asto, R, Kl\'en, Generalized Orlicz spaces and related PDE, Nonlinear Anal. 143, 155-173, (2016). 
\bibitem{hat} A. Hatcher, Algebraic Topology. Cambridge University Press, Cambridge, 2002.
\bibitem{hop} C. Hopper, Partial Regularity for Holonomic Minimizers of Quasiconvex Functionals. Arch. Rational Mech. Anal. 222, 91-141, (2016).


\bibitem {KM1} T. Kuusi, G. Mingione: Guide to nonlinear potential estimates, Bull. Math. Sci. 4, 1-82, (2014). 

\bibitem {KM2} T. Kuusi, G. Mingione: Vectorial nonlinear potential theory, J.~Europ.~Math.~Soc. (JEMS), 20 929-1004, (2018). 



\bibitem{luc} S. Luckhaus, Partial H\"older continuity for minima of certain energies among maps into a Riemannian manifold. Indiana Univ. Math. J. 37, 349-367, (1988).
\bibitem{lie} G. M. Liebermann, Gradient estimates for a new class of degenerate elliptic and parabolic equations. Ann. Sc. Norm. Super. Pisa, Cl. Sci. $4^{e}$ s\'erie, tome 21, no 4, p. 497-522, (1994).

\bibitem {Manth} J. J. Manfredi, Regularity of the gradient for a class of nonlinear possibly
degenerate elliptic equations, Ph.D. Thesis,  University of
Washington, St. Louis. (1986). 

\bibitem{mar1} P. Marcellini, Regularity of minimizers of integrals of the calculus of variations with non standard growth conditions. Arch. Ration. Mech. Anal. 105, 267-284, (1989).
\bibitem{mar2} P. Marcellini, Regularity and existence of solutions of elliptic equations with p,q-growth conditions. J. Differ. Equations 90, 1-30, (1991).
\bibitem{dark} G. Mingione, Regularity of minima: an invitation to the Dark Side of the Calculus of Variations. Applications of Mathematics 51.4: 355-426, (2006). 
\bibitem{RT1} M. A. Ragusa, A. Tachikawa, Boundary regularity of minimizers of $p(x)$-energy functionals, Ann. IHP, Nonlinear Anal. 33 451-476, (2017). 

\bibitem{RT2} M. A. Ragusa, A. Tachikawa, On interior regularity of minimizers of $p(x)$-energy functionals. Nonlinear Anal. 93 162-167, (2013). 
    
    
\bibitem{RT3} M. A. Ragusa, A. Tachikawa, H. Takabayashi, Partial regularity of $p(x)$-harmonic maps. Transaction of the AMS, Vol. 365, nr. 6, 3329-3353, (2013).

\bibitem{RR}
K. R. Rajagopal, M. \Ruz{}, Mathematical modeling
of electrorheological materials, Cont. Mech. 
Therm.,13, 59-78, (2001). 

\bibitem{shouhl} R. Schoen, K. Uhlenbeck, A Regularity Theory for Harmonic Maps. J. Differential Geometry, 17, 307-335, (1982).
\bibitem{sim} L. Simon, Theorems on Regularity and Singularity of Energy Minimizing Maps. Lectures in Mathematics, ETH Z\"urich, Birkh\"auser, Basel, (1996).
\bibitem{tac} A. Tachikawa, On the singular set of $p(x)$-energy. Calc. Var. \& PDE 50, 145-169, (2014).
\bibitem{TU} A. Tachikawa, K. Usuba, Regularity results up to the boundary for minimizers of $p(x)$-energy with $p(x)>1$.  manu. math.152 (2017), 127--151. 
\bibitem{uhl} K. Uhlenbeck, Regularity for a class of non-linear elliptic sysyem. Acta Math. 138, 219-240, (1977).
\bibitem{ura} N. N. Ural'tseva - Degenerate quasilinear elliptic systems. Zap. Na. Sem. Leningrad. Otdel. Mat. Inst. Steklov. (LOMI) 7, 184-222, (1968).
\bibitem{U} K. Usuba, Partial regularity of minimizers of p(x)-growth functionals with $1<p(x)<2$, Bull. Lond. Math. Soc. 47  455-472,  (2015).
\bibitem{zhi1} V. V. Zhikov, On some variational problems. Russian J. Math. Phys. 5, 105-116, (1997).
\bibitem{zhi2} V. V. Zhikov, Averaging of functionals of the calculus of variations and
elasticity theory, Izv. Akad. Nauk SSSR Ser. Mat. 50,
675--710, (1986).

\bibitem{zhi3} V. V. Zhikov, On Lavrentiev's Phenomenon, Russian J. Math. Phys. 3, 249--269, (1995).

\bibitem{zhi4} V. V. Zhikov, S. M.Kozlov, O. A.Oleinik: Homogenization of
differential operators and integral functionals. Springer-Verlag,
Berlin, 1994.


\end{thebibliography}
\end{document}